\documentclass[12pt]{amsart}
\usepackage{hyperref}
\usepackage[margin=25mm]{geometry}
\IfFileExists{lmodern.sty}{\usepackage[T1]{fontenc}\usepackage{lmodern}}{%
 \DeclareTextCommand{\dh}{OT1}{\leavevmode\setbox0\hbox{d}%
 \rlap{\raise.6\ht0\hbox to\wd0{\hss\vrule width.32em height.045em\kern.02em}}d}}
\usepackage{microtype}
\usepackage{amsmath,amssymb,amsthm,mathtools}
\usepackage{enumitem}
\usepackage{hyperref}
\hypersetup{colorlinks=true,linkcolor=black,citecolor=black,urlcolor=blue}
\allowdisplaybreaks[1]
\numberwithin{equation}{section}
\setlist[enumerate]{itemsep=3pt,topsep=4pt}

\newtheorem{theorem}{Theorem}[section]
\newtheorem{proposition}[theorem]{Proposition}
\newtheorem{lemma}[theorem]{Lemma}
\newtheorem{corollary}[theorem]{Corollary}
\newtheorem{mainthm}{Theorem}

\theoremstyle{definition}
\newtheorem{example}[theorem]{Example}
\newtheorem{remark}[theorem]{Remark}

\newcommand{\C}{\mathbb C}
\newcommand{\R}{\mathbb R}
\newcommand{\N}{\mathbb N}
\newcommand{\Q}{\mathbb Q}
\newcommand{\T}{\mathbb T}
\newcommand{\cP}{\mathcal P}
\newcommand{\cB}{\mathcal B}
\newcommand{\cR}{\mathcal R}
\DeclareMathOperator{\Log}{Log}
\DeclareMathOperator{\conv}{conv}
\DeclareMathOperator{\pr}{pr}
\DeclareMathOperator{\spn}{span}

\title{Oka-Weil approximation theorem for polynomials with exponents\\
in compact convex sets in \texorpdfstring{$\mathbb R^n$}{R^n}}
\author{Duong Thi Thuy Van and Kieu Phuong Chi}
\date{\today}

\address{Duong Thi Thuy Van,  \newline
Faculty of Natural Science Education, Dong Nai University, Vietnam}
\email{\href{mailto:vandtt@dnpu.edu.vn}{vandtt@dnpu.edu.vn}}

\address{Kieu Phuong Chi,  \newline
Department of Mathematics and Applications, Saigon  University, 273  An Duong Vuong, Ho Chi Minh City, Vietnam}
\email{\href{mailto:kieuphuongchi@sgu.edu.vn}{kieuphuongchi@sgu.edu.vn}}
\begin{document}

\begin{abstract}
Let $S\subset\R^n_+$ be a compact convex set with $0\in S$, and let
$\cP^S$ be the ring of polynomials whose exponents of the $m$-th degree
lie in the dilates $mS$. Sigurðardóttir proved an Oka--Weil theorem for
compact $\cP^S$-convex subsets of $(\C^*)^n$ and asked whether it holds for
all compact $\cP^S$-convex subsets of $\C^n$. We answer this question
affirmatively when the cone $\R_+S$ is rational polyhedral, for every $S$
when $n=2$, and for several classes of nonpolyhedral cones in all
dimensions; in all these cases the compact sets may meet the coordinate
hyperplanes. For arbitrary $S$ we prove that the Oka--Weil property of a
compact $\cP^S$-convex set is equivalent to the compactness of its hull
with respect to a finitely generated monomial subring of $\cP^S$, to a
corona-type certificate, and to the existence of a bounded neighbourhood
defined by finitely many inequalities $|p_j|<c_j$ with $p_j\in\cP^S$.
The proofs use monomial charts, the classical Oka--Weil theorem in an
auxiliary affine space, and elementary Gelfand theory. As applications we
obtain an Oka--Stolzenberg criterion for $\cP^S$-hulls and the
corresponding rational approximation theorem.
\end{abstract}
\maketitle

\section{Introduction}

The Oka--Weil theorem states that a function holomorphic in a neighbourhood
of a compact polynomially convex set $K\subset\C^n$ can be approximated
uniformly on $K$ by polynomials, and its rational version gives
approximation by rational functions without poles on $K$ when $K$ is
rationally convex. The polynomial version is due to Oka \cite{Oka} and
Weil \cite{Weil}; see \cite{Hor,Stout} for modern expositions of both
versions. Stolzenberg \cite{Stolz} later
characterized the polynomial and the rational hull of a compact set by
families of complex hypersurfaces, and the Oka--Weil theorem is the
analytic ingredient of his criterion. The purpose of this paper is to
establish the Oka--Weil theorem, and then Stolzenberg's criterion, when the
ring of all polynomials is replaced by a ring of polynomials whose
exponents are restricted by a compact convex set.

Let $S\subset\R^n_+$ be compact and convex with $0\in S$. For an integer
$m\ge0$ let $\cP^S_m$ be the space of all polynomials of the form
\[
 p(z)=\sum_{\alpha\in mS\cap\N^n}c_\alpha z^\alpha ,
\]
and let
\[
 \cP^S=\bigcup_{m\ge0}\cP^S_m .
\]
When $S$ is the standard simplex, $\cP^S_m$ is the space of polynomials of
degree at most $m$; other choices of $S$ give sparse polynomial spaces
which are useful in multivariate approximation \cite{BL,Tref}, and a
systematic pluripotential theory for these spaces has been developed by
Magnússon, Sigurðardóttir, Sigurðsson and Snorrason \cite{MSSS,MSS,MSiS}.
Since $\cP^S_k\cP^S_m\subset\cP^S_{k+m}$, the space $\cP^S$ is a ring; it
consists of the polynomials whose exponents lie in the convex cone
$\Gamma_S=\R_+S$, and it is a proper subring of $\C[z_1,\ldots,z_n]$ unless
$S$ is a neighbourhood of $0$ in $\R^n_+$. The hull of a compact set $K$
with respect to $\cP^S$ is
\[
 \widehat K^S=\bigl\{z\in\C^n:\ |p(z)|\le\|p\|_K\ \text{for every }
 p\in\cP^S\bigr\},
 \qquad \|p\|_K=\sup_K|p|,
\]
and $K$ is called \emph{$S$-convex} if $\widehat K^S=K$. The hull
$\widehat K^S$ and the term $S$-convex are those of Sigurðardóttir
\cite{Sig}, and the description of $\cP^S$ as the ring of polynomials with
exponents in the cone $\R_+S$ is taken from \cite[Section~1]{Sig}.

Sigurðardóttir \cite[Theorem~1.1]{Sig} proved, by Hörmander's $L^2$
method with weights constructed from Siciak--Zakharyuta functions, that
every function holomorphic in a neighbourhood of a compact $S$-convex set
$K\subset(\C^*)^n$ is a uniform limit on $K$ of elements of $\cP^S$. She
asked whether the same is true for all compact $S$-convex subsets of
$\C^n$ \cite[Section~1]{Sig}. The restriction to the torus is not a technical one, as the following
example, which is also taken from \cite[Section~1]{Sig}, shows. For
\[
 S=\conv\{0,e_1,e_1+e_2\}\qquad\text{we have}\qquad
 \cP^S=\C[z_1,\,z_1z_2],
\]
so every element of $\cP^S$ is constant on the line $\{z_1=0\}$, and a
compact $S$-convex set cannot meet this line. On the other hand, the set
\[
 \bigl\{z\in\C^2:\ |z_1-2|\le1,\ |z_1z_2|\le1\bigr\}
\]
is compact, $S$-convex, and meets the line $\{z_2=0\}$. The question is
whether the Oka--Weil theorem holds on sets of this kind, which meet some
coordinate hyperplanes but not others.

Our approach rests on an elementary observation. If $z^\beta$ and
$z^{\beta+e_1},\ldots,z^{\beta+e_n}$ belong to $\cP^S$, then on the set
$\{z^\beta\ne0\}$ every coordinate is a quotient of two elements of
$\cP^S$:
\[
 z_j=\frac{z^{\beta+e_j}}{z^\beta},\qquad j=1,\ldots,n.
\]
We call such exponents $\beta$ \emph{good}, denote their set by $G_S$,
and put
\[
 U_S=\bigcup_{\beta\in G_S}\{z\in\C^n:\ z^\beta\ne0\}.
\]
A compact $S$-convex set is always contained in $U_S$
(Corollary~\ref{cor:goodcover}); this is the place where the convexity of
$S$, through the saturation of the semigroup $\R_+S\cap\N^n$, is used. The
remaining difficulty is of a global nature. The $S$-convexity of $K$
controls point evaluations at points of $\C^n$, but the uniform closure of
$\cP^S$ on $K$ may a priori have characters that are not point
evaluations. When $\R_+S$ is rational polyhedral, these characters
correspond to points of the affine toric variety
$\operatorname{Spec}\cP^S$ that do not lie in the image of $\C^n$. We show
that the classical Oka--Weil theorem in an auxiliary space $\C^N$, combined
with the most basic facts of Gelfand theory, excludes such characters for
finitely generated monomial algebras (Theorem~\ref{thm:chart}). Unlike the proof of \cite[Theorem~1.1]{Sig}, the argument does not use
the $L^2$ method. Its two ingredients are classical: the formula for $z_j$ above is the localization of a semigroup
algebra at a monomial, which in the rational polyhedral case describes the
affine charts of $\operatorname{Spec}\cP^S$ \cite[Sections~1.1--1.3]{CLS},
and the passage to an ambient space $\C^N$ through a finite set of
generators is the standard description of the spectrum of a finitely
generated uniform algebra \cite{Gam,Stout}. The good charts allow us to combine the two without any global
information on $\operatorname{Spec}\cP^S$. Our first result answers Sigurðardóttir's question for rational
polyhedral cones.

\begin{mainthm}\label{thmA}
Let $S\subset\R^n_+$ be compact and convex with $0\in S$, and assume that $\R_+S$ is a rational polyhedral cone, that is, the set
of nonnegative linear combinations of finitely many vectors with rational
coordinates. If $K\subset\C^n$ is compact and $S$-convex, then every function holomorphic in a neighbourhood of $K$ can
be approximated uniformly on $K$ by elements of $\cP^S$. Moreover, a
compact set $K$ is $S$-convex if and only if $K\subset U_S$ and the image
of $K$ under a monomial map defined by generators of $\cP^S$ is
polynomially convex.
\end{mainthm}

For general $S$ the ring $\cP^S$ need not be finitely generated, and we
reduce the problem to finite stages. A \emph{sparse polyhedron} around a
compact set $K$ is a set of the form
\begin{equation}\label{eq:polyhedron}
 D=\bigl\{z\in\C^n:\ |p_j(z)|<c_j,\ j=1,\ldots,N\bigr\},
 \qquad p_j\in\cP^S,\quad c_j>\|p_j\|_K .
\end{equation}

\begin{mainthm}\label{thmB}
Let $S\subset\R^n_+$ be compact and convex with $0\in S$, and let
$K\subset\C^n$ be a nonempty compact $S$-convex set. The following
conditions are equivalent.
\begin{enumerate}[label=\textup{(\roman*)}]
\item Every function holomorphic in a neighbourhood of $K$ is a uniform
limit on $K$ of elements of $\cP^S$.
\item There is a finitely generated monomial subring $\cB\subset\cP^S$
such that $\widehat K^{\cB}$ is compact.
\item There are $\beta^1,\ldots,\beta^r\in G_S$ and
$q_1,\ldots,q_r\in\cP^S$ such that
\[
 \Bigl\|\,1-\sum_{k=1}^r z^{\beta^k}q_k\Bigr\|_K<1 .
\]
\item $K$ is contained in a bounded sparse polyhedron.
\end{enumerate}
\end{mainthm}

Condition (iii) is a condition of corona type: by elementary Gelfand
theory it says that the spectrum of the uniform closure of $\cP^S$ on $K$
is $K$ itself, and no point of a ``corona'' outside $K$ is present
(Remark~\ref{rem:open}). Thus Sigurðardóttir's question is equivalent to
the assertion that every
compact $S$-convex set lies in a bounded set defined by finitely many
inequalities $|p_j|<c_j$ with $p_j\in\cP^S$. We verify these conditions in
several cases.

\begin{mainthm}\label{thmC}
Let $S\subset\R^n_+$ be compact and convex with $0\in S$, and let $K$ be a
compact $S$-convex set. The conditions of Theorem~\ref{thmB} hold, and
hence the Oka--Weil theorem holds on $K$, in each of the following cases:
\begin{enumerate}[label=\textup{(\alph*)}]
\item $n=2$;
\item $K\subset(\C^*)^n$;
\item there is a nonempty proper subset $I\subset\{1,\ldots,n\}$ such that
 \begin{itemize}
 \item $e_i\in\R_+S$ for every $i\in I$;
 \item for all $i\in I$ and $k\notin I$ there is an integer $m\ge1$ with
 $me_i+e_k\in\R_+S$;
 \item every nonzero $\alpha\in\R_+S\cap\N^n$ satisfies
 $\alpha-\delta e_i\in\R_+S$ for some $i\in I$ and some $\delta>0$.
 \end{itemize}
\end{enumerate}
Moreover, for arbitrary $n$ and $S$ the Oka--Weil theorem holds on every
compact $S$-convex set contained in a coordinate subspace of dimension at
most two.
\end{mainthm}

Case~(b) is Sigurðardóttir's theorem combined with Theorem~\ref{thmB}; the
new information there is the finite-stage compactness, which is used in
the proof of~(a) and in a block criterion (Proposition~\ref{prop:block})
that applies to cones containing no coordinate vector. Case~(c) contains
closed nonpolyhedral cones such as
\[
 \Bigl\{x\in\R^n_+:\ x_1\ge\pi\bigl(x_2^2+\cdots+x_n^2\bigr)^{1/2}\Bigr\},
\]
and the cones in case~(a) need not be closed
(Example~\ref{ex:nonclosed}). As an application of the Oka--Weil theorem
we obtain the following analogue of Stolzenberg's criterion
\cite[(1.2)]{Stolz}.

\begin{mainthm}\label{thmD}
Let $K\subset\C^n$ be compact and assume that $\widehat K^{\cB}$ is
compact for some finitely generated monomial subring $\cB\subset\cP^S$; by
Theorems~\ref{thmA} and~\ref{thmC} this holds, for instance, for every
compact $S$-convex set when $\R_+S$ is rational polyhedral or $n=2$. A
point $z\in\C^n$ lies outside $\widehat K^S$ if and only if there is a
continuous family $(p_t)_{t\ge0}$ in some $\cP^S_m$ such that
\begin{itemize}
\item $p_0(z)=0$ and no $p_t$ vanishes on $K$;
\item the hypersurfaces $\{p_t=0\}$ leave every compact subset of $\C^n$
as $t\to\infty$.
\end{itemize}
\end{mainthm}

We also prove the rational counterparts of these results
(Theorem~\ref{thm:rational}). The paper is organized as follows.
Section~\ref{sec:prelim} describes the good coordinate locus, and
Section~\ref{sec:chart} contains the chart theorem and the proof of
Theorem~\ref{thmA}. Theorem~\ref{thmB} is proved in
Section~\ref{sec:finite}, Theorem~\ref{thmC} in Section~\ref{sec:cases},
where we also discuss a block criterion, and
Section~\ref{sec:applications} is devoted to the Oka--Stolzenberg
criterion and to rational approximation. In dimensions $n\ge3$ we do not
know whether the conditions of Theorem~\ref{thmB} hold for every compact
$S$-convex set when $\R_+S$ is not rational polyhedral;
Remark~\ref{rem:open} gives an equivalent formulation of this problem in
terms of the spectrum of a uniform algebra.

\section{Exponent cones and the good coordinate locus}\label{sec:prelim}

Throughout, $\N=\{0,1,2,\ldots\}$, $\R_+=[0,\infty)$, and $e_1,\ldots,e_n$
is the standard basis of $\R^n$. A \emph{convex cone} is a convex subset
$\Gamma\subset\R^n_+$ with $0\in\Gamma$ and $t\Gamma\subset\Gamma$ for
$t\ge0$; it is closed under addition, but it need not be closed or
polyhedral. We put
\[
 M_\Gamma=\Gamma\cap\N^n,\qquad
 \cP^\Gamma=\spn\{z^\alpha:\ \alpha\in M_\Gamma\}=\C[M_\Gamma].
\]
If $S\subset\R^n_+$ is compact and convex with $0\in S$, then
$\Gamma_S=\R_+S$ is a convex cone and
\[
 \cP^S=\cP^{\Gamma_S}.
\]
Indeed, if $\alpha\in\Gamma_S$, then $\alpha=ts$ with $s\in S$ and
$t\ge0$, and for every integer $m\ge t$ we have
\[
 \alpha=m\cdot\Bigl(\frac tm\,s\Bigr)\in mS,
\]
because $0\in S$ and $S$ is convex; the converse inclusion is obvious.
All results below are stated for an arbitrary convex cone $\Gamma$; the
case $\Gamma=\Gamma_S$ gives the statements in the introduction. We write
$\widehat K^\Gamma$ for the hull with respect to $\cP^\Gamma$ and call $K$
\emph{$\Gamma$-convex} if $\widehat K^\Gamma=K$. We say that the
\emph{Oka--Weil property} holds on a compact set $K$ if every function
holomorphic in a neighbourhood of $K$ is a uniform limit on $K$ of
elements of $\cP^\Gamma$.

More generally, let $\cB\subset\C[z_1,\ldots,z_n]$ be a unital
subalgebra. For a compact set $K$ we use the polynomial and the rational
hull with respect to $\cB$,
\begin{align*}
 \widehat K^{\cB}&=\bigl\{z\in\C^n:\ |b(z)|\le\|b\|_K\ \text{for all }
 b\in\cB\bigr\},\\
 \cR_\cB(K)&=\bigl\{z\in\C^n:\ b(z)\in b(K)\ \text{for all }b\in\cB\bigr\}.
\end{align*}
Both are closed, $\cR_\cB(K)\subset\widehat K^\cB$, and both hulls become
smaller when $\cB$ becomes larger. For $\cB=\C[z]$ we obtain the ordinary
hulls $\widehat K$ and $\cR(K)$, so that
\[
 \widehat K\subset\widehat K^\cB,\qquad \cR(K)\subset\cR_\cB(K).
\]
A point $z$ lies outside $\cR_\cB(K)$ exactly when some $b\in\cB$ vanishes
at $z$ but has no zero on $K$. A \emph{monomial algebra} is a unital
subalgebra spanned by monomials. For such an algebra we define the set of
\emph{good exponents} and the \emph{good locus} by
\begin{align*}
 G(\cB)&=\bigl\{\beta\in\N^n:\ z^\beta\in\cB\ \text{and}\
 z^{\beta+e_j}\in\cB\ \text{for}\ j=1,\ldots,n\bigr\},\\
 U_\cB&=\bigcup_{\beta\in G(\cB)}\{z\in\C^n:\ z^\beta\ne0\}.
\end{align*}
For $\cB=\cP^\Gamma$ we write $G_\Gamma$ and $U_\Gamma$. On
$\{z^\beta\ne0\}$, $\beta\in G(\cB)$, we have
$z_j=z^{\beta+e_j}/z^\beta$, and this simple identity is the basis of the
paper. The next lemma describes $U_\Gamma$ stratum by stratum. Parts of it are
implicit in \cite[Section~2]{Sig}, where it is shown that every element of
$\cP^S$ is constant on the line $\C e_k$ if $e_k\notin\R_+S$, and on the
coordinate subspace $\C^J$ if $S\cap\R^J=\{0\}$; when $\Gamma$ is rational
polyhedral, the lemma can also be read off from the orbit--cone
correspondence \cite[Chapter~3]{CLS}. We give a self-contained proof. For
$J\subset\{1,\ldots,n\}$ we put
\[
 \R^J=\{x\in\R^n:\ x_i=0\ \text{for}\ i\notin J\},\qquad
 M_J=M_\Gamma\cap\R^J .
\]

\begin{lemma}\label{lem:strata}
Let $\Gamma\subset\R^n_+$ be a convex cone, let $z\in\C^n$, and let
$J=\{j:\ z_j\ne0\}$. The following conditions are equivalent.
\begin{enumerate}[label=\textup{(\arabic*)}]
\item There is $\beta\in G_\Gamma\cap\R^J$; equivalently, $z\in U_\Gamma$.
\item $M_J$ spans $\R^J$, and for every $i\notin J$ there is
$\alpha\in M_\Gamma\cap\R^{J\cup\{i\}}$ with $\alpha_i>0$.
\item The fibre
\[
 F(z)=\bigl\{w\in\C^n:\ w^\alpha=z^\alpha\ \text{for all}\
 \alpha\in M_\Gamma\bigr\}
\]
is bounded.
\end{enumerate}
If these conditions hold, then $F(z)=\{z\}$.
\end{lemma}

\begin{proof}
A good exponent $\beta$ satisfies $z^\beta\ne0$ if and only if
$\beta\in\R^J$, which gives the equivalence inside (1). If (1) holds and
$w\in F(z)$, then $w^\beta=z^\beta\ne0$ and
\[
 w_j=\frac{w^{\beta+e_j}}{w^\beta}=\frac{z^{\beta+e_j}}{z^\beta}=z_j
 \qquad(j=1,\ldots,n),
\]
so $F(z)=\{z\}$ and (3) holds.

Assume that (2) fails. If $M_J$ does not span $\R^J$, choose a nonzero
$v\in\R^J$ orthogonal to $M_J$ and put, for $t\in\R$,
\[
 w(t)_j=z_je^{tv_j}\quad(j\in J),\qquad w(t)_i=0\quad(i\notin J).
\]
If $\alpha\in M_\Gamma$ is supported in $J$, then
$w(t)^\alpha=z^\alpha e^{t\langle\alpha,v\rangle}=z^\alpha$; otherwise
both $w(t)^\alpha$ and $z^\alpha$ vanish. Thus the unbounded curve $w(t)$
lies in $F(z)$. If for some $i\notin J$ no
$\alpha\in M_\Gamma\cap\R^{J\cup\{i\}}$ has $\alpha_i>0$, then every
$\alpha\in M_\Gamma$ either has $\alpha_i=0$ or has a positive coordinate
outside $J\cup\{i\}$, and in both cases $w\mapsto w^\alpha$ is constant on
the line $z+\C e_i$; hence this line lies in $F(z)$. Therefore (3)
implies (2).

Assume (2). Choose linearly independent
$\gamma^1,\ldots,\gamma^{|J|}\in M_J$, let $C\subset\Gamma$ be the cone
they generate, and let $\gamma=\gamma^1+\cdots+\gamma^{|J|}$ (with
$\gamma=0$ if $J=\varnothing$). Then $\gamma$ lies in the interior of $C$ relative to $\R^J$: in the
basis $\gamma^1,\ldots,\gamma^{|J|}$ of $\R^J$ the cone $C$ is the
positive orthant, and $\gamma$ has all coordinates equal to $1$. Since
$\gamma+e_j/k\to\gamma$ as $k\to\infty$, for $j\in J$ we have
\[
 k\gamma+e_j=k\Bigl(\gamma+\frac{e_j}k\Bigr)\in C
 \qquad\text{for all large }k .
\]
For $i\notin J$ let $\alpha$ be as in (2) and write
$\alpha/\alpha_i=e_i+v_i$ with $v_i\in\R^J$. For large $k$ we have $k\gamma-v_i=k(\gamma-v_i/k)\in C$ for the same
reason, and hence, since $\Gamma$ is closed under addition and
$\alpha/\alpha_i\in\Gamma$,
\[
 k\gamma+e_i=(k\gamma-v_i)+\frac{\alpha}{\alpha_i}\in\Gamma .
\]
Taking one large integer $k$ for all indices, we obtain
$\beta=k\gamma\in G_\Gamma\cap\R^J$.
\end{proof}

\begin{corollary}\label{cor:goodcover}
Let $\Gamma\subset\R^n_+$ be a convex cone and let $K\subset\C^n$ be a
nonempty compact set with $\widehat K^\Gamma=K$ or $\cR_\Gamma(K)=K$.
Then $K\subset U_\Gamma$, so that there are finitely many
$\beta^1,\ldots,\beta^r\in G_\Gamma$ with
\[
 K\subset\bigcup_{k=1}^r\{z^{\beta^k}\ne0\}.
\]
Moreover, $\Gamma$ has nonempty interior, and $0\notin K$ unless
$\Gamma=\R^n_+$.
\end{corollary}

\begin{proof}
Every element of $\cP^\Gamma$ is constant on $F(z)$, so $F(z)$ is
contained in both hulls of $K$ when $z\in K$. Hence $F(z)$ is bounded and
Lemma~\ref{lem:strata} gives $K\subset U_\Gamma$; the finite cover follows
from compactness. If $\beta\in G_\Gamma$, then $e_j=(\beta+e_j)-\beta$
lies in the linear span of $M_\Gamma$ for every $j$, so $\Gamma$ has
nonempty interior. If $0\in K$, then condition (2) of
Lemma~\ref{lem:strata} with $J=\varnothing$ gives $e_i\in\Gamma$ for all
$i$, so $\Gamma=\R^n_+$.
\end{proof}

In particular, nonempty compact $\Gamma$-convex sets exist only for cones
with nonempty interior, and it costs nothing to assume this.

\begin{example}\label{ex:triangle}
Let $S=\conv\{0,e_1,e_1+e_2\}$. Then
\[
 \Gamma_S=\{x\in\R^2_+:\ x_2\le x_1\},\qquad
 \cP^S=\C[z_1,\,z_1z_2],
\]
since $M_{\Gamma_S}$ is generated by $(1,0)$ and $(1,1)$. A pair
$\beta=(a,b)$ is good exactly when $b<a$, so $U_S=\{z_1\ne0\}$. For $0<r<2$
and $s>0$ the map $(z_1,z_2)\mapsto(z_1,z_1z_2)$ identifies the set
\[
 K_{r,s}=\bigl\{z\in\C^2:\ |z_1-2|\le r,\ |z_1z_2|\le s\bigr\}
\]
with a closed bidisc, so $K_{r,s}$ is compact and $S$-convex, and it meets
$\{z_2=0\}$. For the unit torus $\T^2$, however,
\[
 \widehat{\T^2}^{\,S}=\bigl\{z\in\C^2:\ |z_1|\le1,\ |z_1z_2|\le1\bigr\}
\]
contains the line $\{z_1=0\}$; the hull of a compact set need not be
bounded. This example is taken from \cite[Section~1]{Sig}, where it is
also observed that the hull of a single point $(0,z_2)$ is the whole line
$\{z_1=0\}$.
\end{example}

\section{Monomial charts and finitely generated exponent semigroups}
\label{sec:chart}

Let $\cB$ be a finitely generated monomial algebra, generated by $1$ and
monomials $g_1,\ldots,g_N$, and let
\[
 \pi=(g_1,\ldots,g_N):\C^n\longrightarrow\C^N .
\]
Every $b\in\cB$ has the form $b=Q\circ\pi$ with a polynomial $Q$ on
$\C^N$, and therefore
\begin{equation}\label{eq:pullback}
 \widehat K^\cB=\pi^{-1}\bigl(\widehat{\pi(K)}\bigr),\qquad
 \cR_\cB(K)=\pi^{-1}\bigl(\cR(\pi(K))\bigr).
\end{equation}
For a compact set $K$ we denote by $A_\cB(K)$ the uniform closure of
$\cB|_K$ in $C(K)$ and by $A^{\mathrm{rat}}_\cB(K)$ the uniform closure of
the set of quotients
\[
 \frac bc,\qquad b,c\in\cB,\quad c\ne0\ \text{on}\ K .
\]
Both are closed unital subalgebras of $C(K)$. A \emph{character} of such
an algebra $A$ is a nonzero multiplicative linear functional
$\psi:A\to\C$. We use only the following basic facts from Gelfand theory,
see \cite[Chapter~I]{Gam}.
\begin{enumerate}[label=\textup{(G\arabic*)}]
\item Every character satisfies $|\psi(f)|\le\|f\|_K$ for $f\in A$; in
particular it is continuous.
\item Every proper ideal of $A$ is contained in a maximal ideal, and every
maximal ideal is the kernel of a character.
\item Consequently, if $f_1,\ldots,f_r\in A$ and no character vanishes at
all of $f_1,\ldots,f_r$, then there are $h_1,\ldots,h_r\in A$ with
$f_1h_1+\cdots+f_rh_r=1$. In particular, an element on which no
character vanishes is invertible in $A$.
\end{enumerate}
Indeed, (G3) follows from (G2), because the ideal generated by
$f_1,\ldots,f_r$ cannot be proper.

The next lemma is the localization identity
$\cB[(z^\beta)^{-1}]=\C[z_1,\ldots,z_n][(z^\beta)^{-1}]$ for semigroup
algebras, familiar from toric geometry \cite[Sections~1.1--1.3]{CLS},
written in terms of characters.

\begin{lemma}\label{lem:chartlemma}
Let $\cB$ be a monomial algebra and let $\varphi:\cB\to\C$ be a unital
homomorphism with $\varphi(z^\beta)\ne0$ for some $\beta\in G(\cB)$. Put
\[
 z_j=\frac{\varphi(z^{\beta+e_j})}{\varphi(z^\beta)},\qquad j=1,\ldots,n.
\]
Then $z^\beta\ne0$, and $\varphi(b)=b(z)$ for every $b\in\cB$.
\end{lemma}

\begin{proof}
Let $z^\alpha\in\cB$. Since
$\alpha+|\alpha|\beta=\sum_j\alpha_j(\beta+e_j)$, we have the identity
\[
 z^\alpha\,(z^\beta)^{|\alpha|}=\prod_{j=1}^n(z^{\beta+e_j})^{\alpha_j}
\]
between products of elements of $\cB$; here
$|\alpha|=\alpha_1+\cdots+\alpha_n$, and all factors belong to $\cB$
because $z^\alpha\in\cB$ and $\beta\in G(\cB)$. Since $\varphi$ is
multiplicative and $\varphi(z^{\beta+e_j})=z_j\,\varphi(z^\beta)$ by the
definition of $z_j$, applying $\varphi$ gives
\[
 \varphi(z^\alpha)\,\varphi(z^\beta)^{|\alpha|}
 =\prod_{j=1}^n\bigl(z_j\,\varphi(z^\beta)\bigr)^{\alpha_j}
 =z^\alpha\,\varphi(z^\beta)^{|\alpha|},
\]
hence $\varphi(z^\alpha)=z^\alpha$. For $\alpha=\beta$ this gives
$z^\beta=\varphi(z^\beta)\ne0$, and the assertion follows by linearity.
\end{proof}

The following lemma is a classical consequence of the Oka--Weil theorem;
in the language of uniform algebras it is a special case of Shilov's
idempotent theorem \cite{Gam}. We include the short proof.

\begin{lemma}\label{lem:nopieces}
Let $L\subset\C^N$ be compact and let $E$ be a compact set disjoint from
$L$. If $\widehat L=L\cup E$ or $\cR(L)=L\cup E$, then $E=\varnothing$.
\end{lemma}

\begin{proof}
Let $X$ be $\widehat L$ or $\cR(L)$; it is compact and, respectively,
polynomially or rationally convex. Since $L$ and $E$ are disjoint compact
sets, they have disjoint open neighbourhoods $V_L$ and $V_E$, and the
function equal to $0$ on $V_L$ and to $1$ on $V_E$ is holomorphic on the
neighbourhood $V_L\cup V_E$ of $X$. By the Oka--Weil theorem or
its rational version \cite[Ch.~1]{Stout} there is a polynomial,
respectively a rational function $G=P/Q$ with $Q\ne0$ on $X$, such that
\[
 |G|<\tfrac14\ \text{on}\ L,\qquad |G|>\tfrac34\ \text{on}\ E .
\]
In the polynomial case every $y\in E$ satisfies
$|G(y)|\le\|G\|_L<1/4$ because $y\in\widehat L$, a contradiction. In the
rational case let $y\in E$ and $c=G(y)$. The polynomial $P-cQ$ vanishes at
$y\in\cR(L)$, hence at some $x\in L$, and $G(x)=c$ because $Q(x)\ne0$;
this contradicts $|c|>3/4>|G(x)|$.
\end{proof}

The proof of the next theorem follows the classical identification of
the characters of a finitely generated uniform algebra with the points of
the polynomial, or rational, hull of the image of the generators
\cite{Gam,Stout}. The good charts allow us to identify this hull with
$\pi(K)$ itself.

\begin{theorem}[Chart theorem]\label{thm:chart}
Let $\cB$ be a finitely generated monomial algebra, and let
$K\subset\C^n$ be a compact set with $K\subset U_\cB$.
\begin{enumerate}[label=\textup{(\alph*)}]
\item If $\widehat K^\cB=K$, then $\pi(K)$ is polynomially convex, every
character of $A_\cB(K)$ is the evaluation at a point of $K$, and there
are $\beta^1,\ldots,\beta^r\in G(\cB)$ and $q_1,\ldots,q_r\in\cB$ with
\begin{equation}\label{eq:certificate}
 \Bigl\|\,1-\sum_{k=1}^r z^{\beta^k}q_k\Bigr\|_K<1 .
\end{equation}
Every function holomorphic in a neighbourhood of $K$ belongs to
$A_\cB(K)$.
\item If $\cR_\cB(K)=K$, then $\pi(K)$ is rationally convex, and every
function holomorphic in a neighbourhood of $K$ belongs to
$A^{\mathrm{rat}}_\cB(K)$.
\end{enumerate}
\end{theorem}

\begin{proof}
Choose $\beta^1,\ldots,\beta^r\in G(\cB)$ such that the sets
$\{z^{\beta^k}\ne0\}$ cover $K$. Enlarging the set of generators we may
assume that $z^{\beta^k}=g_{i_k}$, so that
\[
 z^{\beta^k}=w_{i_k}\circ\pi,\qquad k=1,\ldots,r,
\]
where $w_1,\ldots,w_N$ are the coordinates of $\C^N$. Put $L=\pi(K)$.

(a) Let $W=\widehat L$. For $w\in W$ and $b=Q\circ\pi\in\cB$ put
$\varphi_w(b)=Q(w)$. This is well defined, since $Q\circ\pi=Q'\circ\pi$
implies
\[
 |(Q-Q')(w)|\le\|Q-Q'\|_L=0 .
\]
Thus $\varphi_w$ is a unital homomorphism with $|\varphi_w(b)|\le\|b\|_K$.
Let
\[
 E=\bigl\{w\in W:\ w_{i_1}=\cdots=w_{i_r}=0\bigr\}.
\]
If $w\in W\setminus E$, then $\varphi_w(z^{\beta^k})\ne0$ for some $k$,
and Lemma~\ref{lem:chartlemma} gives a point $z$ with $\varphi_w(b)=b(z)$
for all $b\in\cB$. Then $|b(z)|\le\|b\|_K$ for all $b$, so
$z\in\widehat K^\cB=K$, and
\[
 w=\bigl(\varphi_w(g_1),\ldots,\varphi_w(g_N)\bigr)=\pi(z)\in L .
\]
Hence $W=L\cup E$. The set $E$ is closed in the compact set $W$, being a
common zero set of coordinate functions, and hence compact. It is disjoint
from $L$, because every point of $K$ has some $z^{\beta^k}\ne0$. Therefore
Lemma~\ref{lem:nopieces} gives $W=L$.

Let $\psi$ be a character of $A=A_\cB(K)$ and put
$w=(\psi(g_1),\ldots,\psi(g_N))$. For every polynomial $Q$ on $\C^N$ we
have $Q(w)=\psi(Q\circ\pi)$, hence
\[
 |Q(w)|\le\|Q\circ\pi\|_K=\|Q\|_L ,
\]
and $w\in W=L$. Thus $w=\pi(z)$ with $z\in K$, and $\psi$ agrees with the
evaluation at $z$ on the generators, hence on $\cB$ and, by continuity, on
$A$.

Since $K$ is covered by the sets $\{z^{\beta^k}\ne0\}$ and every
character is the evaluation at a point of $K$, no character of $A$
vanishes on all the functions $z^{\beta^k}$. By (G3) there are
$h_k\in A$ with
\[
 \sum_{k=1}^r z^{\beta^k}h_k=1\qquad\text{on }K .
\]
Choose $\eta>0$ with $\eta\sum_k\|z^{\beta^k}\|_K<1$ and $q_k\in\cB$
with $\|h_k-q_k\|_K<\eta$. Then
\[
 \Bigl\|\,1-\sum_{k=1}^r z^{\beta^k}q_k\Bigr\|_K
 =\Bigl\|\sum_{k=1}^r z^{\beta^k}(h_k-q_k)\Bigr\|_K
 \le\eta\sum_{k=1}^r\|z^{\beta^k}\|_K<1,
\]
which is \eqref{eq:certificate}. Moreover, on $K$ we have
\[
 z_j=\sum_{k=1}^r z^{\beta^k+e_j}h_k,\qquad z^{\beta^k+e_j}\in\cB,
\]
so every coordinate function belongs to the algebra $A$, and hence every
polynomial does.
Since $\widehat K\subset\widehat K^\cB=K$, the set $K$ is polynomially
convex, and the classical Oka--Weil theorem shows that every function
holomorphic near $K$ lies in the closure of the polynomials on $K$, which
is contained in $A$.

(b) Let $R=\cR(L)$. For $w\in R$ and $b=Q\circ\pi$ put again
$\varphi_w(b)=Q(w)$; if $Q\circ\pi=Q'\circ\pi$, then $Q-Q'$ vanishes on
$L$, so that $(Q-Q')(w)\in(Q-Q')(L)=\{0\}$. If $\varphi_w(z^{\beta^k})\ne0$
for some $k$, Lemma~\ref{lem:chartlemma} gives $z$ with
$\varphi_w(b)=b(z)$ for all $b$, and
\[
 b(z)=Q(w)\in Q(L)=b(K)\qquad\text{for all }b=Q\circ\pi\in\cB ,
\]
so $z\in\cR_\cB(K)=K$ and $w=\pi(z)\in L$. As before, $R=L\cup E$ with $E$
compact and disjoint from $L$, and Lemma~\ref{lem:nopieces} gives $R=L$.

Let $\psi$ be a character of $A^{\mathrm{rat}}=A^{\mathrm{rat}}_\cB(K)$
and $w=(\psi(g_1),\ldots,\psi(g_N))$. If $Q$ is a polynomial with
$Q(w)=0$, then $Q$ has a zero on $L$. Otherwise $b=Q\circ\pi$ has no zero
on $K$, so that $1/b\in A^{\mathrm{rat}}$ and
\[
 1=\psi(b)\,\psi(1/b)=Q(w)\,\psi(1/b)=0 .
\]
Applying this to $Q-Q(w)$ we get $w\in R=L$, so $w=\pi(z)$ with $z\in K$.
Then $\psi(b)=b(z)$ on $\cB$, and $\psi(b/c)=b(z)/c(z)$ for every
admissible quotient because $\psi(b)=\psi(c)\psi(b/c)$; by continuity
$\psi$ is the evaluation at $z$. The argument of~(a) now shows that the
coordinate functions belong to $A^{\mathrm{rat}}$. A polynomial $q$ without zeros on $K$ therefore belongs to
$A^{\mathrm{rat}}$, and no character vanishes on it because characters
are point evaluations at points of $K$; by (G3), $1/q\in A^{\mathrm{rat}}$.
Thus $A^{\mathrm{rat}}$ contains every rational function without poles on
$K$. Finally
$\cR(K)\subset\cR_\cB(K)=K$, and the rational Oka--Weil theorem
\cite[Ch.~1]{Stout} completes the proof.
\end{proof}

In the situation of Theorem~\ref{thm:chart}, the possible obstruction is
the set $E$ of points of $\widehat{\pi(K)}$ at which all good monomials
vanish. When $\cB=\cP^\Gamma$ for a rational polyhedral cone, $E$ lies in
the closed subset of the toric variety $\overline{\pi(\C^n)}$ that is not
covered by the charts $\{z^\beta\ne0\}$, $\beta\in G_\Gamma$; this subset
may contain points without preimage in $\C^n$ and may contain singular
points (Example~\ref{ex:threefaces}). The proof shows that such points
never lie in the hull, and it never uses the geometry of the toric
variety.

\begin{theorem}\label{thm:A}
Let $\Gamma\subset\R^n_+$ be a rational polyhedral cone, let
$g_1,\ldots,g_N$ be monomial generators of $\cP^\Gamma$, and let
$\pi=(g_1,\ldots,g_N)$. For a compact set $K\subset\C^n$ the following
hold.
\begin{enumerate}[label=\textup{(\alph*)}]
\item $K$ is $\Gamma$-convex if and only if $K\subset U_\Gamma$ and
$\pi(K)$ is polynomially convex. In this case the Oka--Weil property holds
on $K$.
\item If $\cR_\Gamma(K)=K$, then every function holomorphic in a
neighbourhood of $K$ is a uniform limit on $K$ of quotients $p/q$ with
$p,q\in\cP^\Gamma$ and $q$ without zeros on $K$.
\end{enumerate}
In particular, Theorem~\ref{thmA} holds.
\end{theorem}

\begin{proof}
We first show that $\cP^\Gamma$ is a finitely generated monomial algebra.
After clearing denominators, $\Gamma$ is the set of nonnegative linear
combinations of finitely many integer vectors $v_1,\ldots,v_m$. Let
$\alpha\in M_\Gamma$ and write $\alpha=\sum_lt_lv_l$ with $t_l\ge0$.
Splitting each $t_l$ into its integer part $\lfloor t_l\rfloor$ and its
fractional part $t_l-\lfloor t_l\rfloor\in[0,1)$, we get
\[
 \alpha=\sum_{l=1}^m\lfloor t_l\rfloor\,v_l+\gamma,\qquad
 \gamma=\sum_{l=1}^m\bigl(t_l-\lfloor t_l\rfloor\bigr)v_l .
\]
The vector $\gamma=\alpha-\sum_l\lfloor t_l\rfloor v_l$ has integer
coordinates and lies in $\Gamma\subset\R^n_+$, so $\gamma\in M_\Gamma$;
moreover $\gamma$ lies in the bounded set
$\Pi=\{\sum_ls_lv_l:\ 0\le s_l<1\}$, which contains only finitely many
lattice points. Hence $M_\Gamma$ is generated, as a semigroup, by
$v_1,\ldots,v_m$ and the finitely many points of $M_\Gamma\cap\Pi$, and
$\cP^\Gamma$ is generated by the corresponding monomials. (This is the
usual proof of Gordan's lemma \cite[Proposition~1.2.17]{CLS}.) If $K$ is $\Gamma$-convex, Corollary~\ref{cor:goodcover} gives
$K\subset U_\Gamma$, and Theorem~\ref{thm:chart}(a) shows that $\pi(K)$ is
polynomially convex and that the Oka--Weil property holds. Conversely,
assume that $K\subset U_\Gamma$ and that $\pi(K)$ is polynomially convex.
By \eqref{eq:pullback},
\[
 \widehat K^\Gamma=\pi^{-1}\bigl(\pi(K)\bigr).
\]
If $\pi(z)=\pi(z')$ with $z'\in K$, then every element of $\cP^\Gamma$
takes the same value at $z$ and at $z'$. Choosing $\beta\in G_\Gamma$ with
$z'^\beta\ne0$ we get $z^\beta=z'^\beta\ne0$ and
\[
 z_j=\frac{z^{\beta+e_j}}{z^\beta}=\frac{z'^{\beta+e_j}}{z'^\beta}=z'_j
 \qquad(j=1,\ldots,n).
\]
Hence $\widehat K^\Gamma=K$. Part (b) follows in the same way from
Corollary~\ref{cor:goodcover} and Theorem~\ref{thm:chart}(b).
\end{proof}

The hypothesis $K\subset U_\cB$ in Theorem~\ref{thm:chart} cannot be
omitted, and for algebras that do not come from cones it is not
automatic.

\begin{example}\label{ex:cusp}
This classical example shows what goes wrong without good charts. Let $\cB=\C[z^2,z^3]\subset\C[z]$ and let $K$ be the closed unit disc.
Since $|z^2|>1$ outside $K$, we have $\widehat K^\cB=K$. Every element of
$\cB$ has zero derivative at the origin, and so has every uniform limit on
$K$; hence $f(z)=z$ cannot be approximated. Here
\[
 G(\cB)=\{2,3,4,\ldots\},\qquad U_\cB=\C^*,
\]
so $K\not\subset U_\cB$. The semigroup $\{0,2,3,\ldots\}$ is not of the
form $\Gamma\cap\N$, which is why Corollary~\ref{cor:goodcover} does not
apply.
\end{example}

\begin{example}\label{ex:threefaces}
Let $\Gamma=\{x\in\R^3_+:\ x_3\le x_1+x_2\}$. Then
\[
 \cP^\Gamma=\C[z_1,\,z_2,\,z_1z_3,\,z_2z_3],\qquad
 \pi=(z_1,\,z_2,\,z_1z_3,\,z_2z_3),
\]
and $\pi(\C^3)$ is dense in the quadric cone
$Y=\{w_1w_4=w_2w_3\}\subset\C^4$, a standard example of an affine toric
variety \cite[Chapter~1]{CLS}, which is singular at the origin.
Moreover
\[
 G_\Gamma=\{\beta\in\N^3:\ \beta_3<\beta_1+\beta_2\},\qquad
 U_\Gamma=\{z\in\C^3:\ (z_1,z_2)\ne(0,0)\}.
\]
The points $(0,0,w_3,w_4)\in Y$ with $(w_3,w_4)\ne0$ have no preimage in
$\C^3$, and the corresponding characters of $\cP^\Gamma$ vanish on all
good monomials. By Theorem~\ref{thm:A}, for $0<r<2$ and $s>0$ the compact
set
\[
 \bigl\{z\in\C^3:\ z_1=0,\ |z_2-2|\le r,\ |z_2z_3|\le s\bigr\},
\]
whose image is a bidisc in the plane $\{w_1=w_3=0\}$, is $\Gamma$-convex
and meets $\{z_1=0\}$ and $\{z_3=0\}$. The two-point set
$\{(1,0,0),(0,1,0)\}$ is $\Gamma$-convex and meets each of the three
coordinate hyperplanes, while no compact $\Gamma$-convex set meets the axis
$\{z_1=z_2=0\}$, whose points have unbounded fibres.
\end{example}

\section{Arbitrary cones: reduction to finite stages}\label{sec:finite}

We now allow an arbitrary convex cone $\Gamma$, so that $\cP^\Gamma$ may
be infinitely generated. Sparse polyhedra around $K$ are defined by
\eqref{eq:polyhedron} with $\cP^\Gamma$ in place of $\cP^S$. The next
theorem contains Theorem~\ref{thmB}.

In the language of uniform algebras, condition (ii) of the next theorem
says that the uniform closure of $\cP^\Gamma$ on $K$ is natural on $K$,
that is, its spectrum is $K$ \cite{Gam}.

\begin{theorem}\label{thm:B}
Let $\Gamma\subset\R^n_+$ be a convex cone and let $K\subset\C^n$ be a
nonempty compact $\Gamma$-convex set. Let $A_\Gamma(K)$ be the uniform
closure of $\cP^\Gamma|_K$. The following conditions are equivalent.
\begin{enumerate}[label=\textup{(\roman*)}]
\item The Oka--Weil property holds on $K$.
\item Every character of $A_\Gamma(K)$ is the evaluation at a point of
$K$.
\item There are $\beta^1,\ldots,\beta^r\in G_\Gamma$ and
$q_1,\ldots,q_r\in\cP^\Gamma$ with
\[
 \Bigl\|\,1-\sum_{k=1}^r z^{\beta^k}q_k\Bigr\|_K<1 .
\]
\item $K$ is contained in a bounded sparse polyhedron.
\item There is a finitely generated monomial algebra
$\cB_0\subset\cP^\Gamma$ such that $\widehat K^{\cB_0}$ is compact.
\end{enumerate}
If these conditions hold, then every neighbourhood of $K$ contains a
bounded sparse polyhedron around $K$.
\end{theorem}

\begin{proof}
(i)$\Rightarrow$(ii). By (i), $A_\Gamma(K)$ contains all polynomials, so
it is the uniform closure $P(K)$ of the polynomials on $K$. If $\psi$ is a
character and $w=(\psi(z_1),\ldots,\psi(z_n))$, then
\[
 |q(w)|=|\psi(q)|\le\|q\|_K\qquad\text{for every polynomial }q,
\]
so $w\in\widehat K\subset\widehat K^\Gamma=K$. Thus $\psi$ is the
evaluation at $w$ on the polynomials, and hence, by the continuity (G1)
of $\psi$, on $P(K)$.

(ii)$\Rightarrow$(iii). By Corollary~\ref{cor:goodcover}, $K$ is covered
by finitely many sets $\{z^{\beta^k}\ne0\}$ with $\beta^k\in G_\Gamma$. By
(ii), no character of $A_\Gamma(K)$ vanishes on all $z^{\beta^k}$, and (G3) gives $h_k\in A_\Gamma(K)$ with $\sum_kz^{\beta^k}h_k=1$ on
$K$. Choosing $q_k\in\cP^\Gamma$ with $\|h_k-q_k\|_K<\eta$, where
$\eta\sum_k\|z^{\beta^k}\|_K<1$, we obtain (iii) exactly as in the proof
of Theorem~\ref{thm:chart}.

(iii)$\Rightarrow$(iv). Put
\[
 F=\sum_{k=1}^r z^{\beta^k}q_k,\qquad \vartheta=\|1-F\|_K<1,
\]
and consider the set $D$ of all $z\in\C^n$ such that
\[
 |1-F(z)|<\frac{1+\vartheta}2,\qquad
 |q_k(z)|<\|q_k\|_K+1,\qquad
 |z^{\beta^k+e_j}|<\|z^{\beta^k+e_j}\|_K+1
\]
for all $k$ and $j$. It is a sparse polyhedron around $K$: the defining
functions $1-F$, $q_k$ and $z^{\beta^k+e_j}$ belong to $\cP^\Gamma$, and
each bound exceeds the norm on $K$ of the corresponding function, since
$(1+\vartheta)/2>\vartheta$. On $D$ the first inequality gives
$|F(z)|\ge1-|1-F(z)|>(1-\vartheta)/2$, and the second gives
$|F(z)|\le\sum_k|z^{\beta^k}|\,|q_k(z)|\le C\max_k|z^{\beta^k}|$. Thus
\[
 \delta:=\frac{1-\vartheta}2<|F(z)|\le C\max_k|z^{\beta^k}|,
 \qquad C=\sum_{k=1}^r\bigl(\|q_k\|_K+1\bigr),
\]
so some $k$ satisfies $|z^{\beta^k}|\ge\delta/C$, and then
\begin{equation}\label{eq:coordbound}
 |z_j|=\frac{|z^{\beta^k+e_j}|}{|z^{\beta^k}|}
 \le\frac C\delta\,\max_l\bigl(\|z^{\beta^l+e_j}\|_K+1\bigr),
 \qquad j=1,\ldots,n.
\end{equation}
Hence $D$ is bounded.

(iv)$\Rightarrow$(v). Let $\cB_0$ be generated by the monomials occurring
in the polynomials $p_j$ that define the polyhedron $D$. Every
$z\in\widehat K^{\cB_0}$ satisfies
\[
 |p_j(z)|\le\|p_j\|_K<c_j\qquad(j=1,\ldots,N),
\]
so the closed set $\widehat K^{\cB_0}$ lies in the bounded set $D$.

(v)$\Rightarrow$(i). Let $f$ be holomorphic in an open set
$\Omega\supset K$. By Corollary~\ref{cor:goodcover} there are
$\beta^1,\ldots,\beta^r\in G_\Gamma$ with
\[
 K\subset V:=\bigcup_{k=1}^r\{z^{\beta^k}\ne0\},
\]
and we may assume $\Omega\subset V$. Enumerate
$M_\Gamma=\{\alpha^1,\alpha^2,\ldots\}$ and let $\cB_j$ be the algebra
generated by $\cB_0$, the monomials $z^{\beta^k}$ and $z^{\beta^k+e_i}$
for all $k$ and $i$, and $z^{\alpha^1},\ldots,z^{\alpha^j}$. The sets
\[
 K_j=\widehat K^{\cB_j},\qquad j=1,2,\ldots,
\]
are closed subsets of the compact set $\widehat K^{\cB_0}$, they
decrease, and
\[
 \bigcap_jK_j=\widehat K^\Gamma=K,
\]
because $\bigcup_j\cB_j=\cP^\Gamma$. Hence $K_j\subset\Omega$ for some
$j$: otherwise the sets $K_j\setminus\Omega$ would form a decreasing
sequence of nonempty compact sets, and a point in their intersection
would be a point of $K\setminus\Omega$. For $b\in\cB_j$ we have
$\|b\|_{K_j}\le\|b\|_K$ by the definition of $K_j$, and
$\|b\|_K\le\|b\|_{K_j}$ because $K\subset K_j$; hence
$\widehat{K_j}^{\cB_j}=\widehat K^{\cB_j}=K_j$, that is, $K_j$ is
$\cB_j$-convex, and $K_j\subset V\subset U_{\cB_j}$ because every
$\beta^k$ belongs to $G(\cB_j)$. Theorem~\ref{thm:chart}(a) shows that $f$
is a uniform limit on $K_j$, and hence on $K$, of elements of
$\cB_j\subset\cP^\Gamma$.

For the last assertion, let $D$ be the bounded polyhedron constructed in
the proof of (iii)$\Rightarrow$(iv) and let $\Omega\supset K$ be open.
For each $w$ in the compact set $\overline D\setminus\Omega$ there are
$p_w\in\cP^\Gamma$ and $t_w>0$ with
\[
 |p_w(w)|>t_w>\|p_w\|_K ,
\]
because $w\notin K=\widehat K^\Gamma$. Finitely many of the open sets
$\{|p_w|>t_w\}$ cover $\overline D\setminus\Omega$, and intersecting $D$
with the corresponding sets $\{|p_w|<t_w\}$ gives a bounded sparse
polyhedron $D'$ around $K$. The closure of $D'$ is contained in
$\overline D\cap\bigcap\{|p_w|\le t_w\}$, which does not meet
$\overline D\setminus\Omega$; hence $\overline{D'}\subset\Omega$. This is
the sparse analogue of the classical fact that a polynomially convex set
has a basis of neighbourhoods consisting of polynomial polyhedra
\cite[Section~2.7]{Hor}.
\end{proof}

\begin{remark}\label{rem:nonconvex}
If $K$ is an arbitrary compact set and $H=\widehat K^\Gamma$, then
\[
 \|p\|_H=\|p\|_K\qquad\text{for all }p\in\cP^\Gamma .
\]
Hence conditions (iii)--(v) hold for $K$ if and only if they hold for $H$,
and in this case $H$ is compact and the Oka--Weil property holds on $H$.
In particular, condition (v) for $K$ alone implies that
$\widehat K^\Gamma$ is compact.
\end{remark}

\begin{remark}\label{rem:open}
Let $K$ be a compact $\Gamma$-convex set. By Lemma~\ref{lem:chartlemma}, a
character of $A_\Gamma(K)$ that does not vanish on all good monomials is
the evaluation at a point $z$ with $|b(z)|\le\|b\|_K$ for all
$b\in\cP^\Gamma$, that is, at a point of $K$. Therefore the Oka--Weil
property on $K$ fails exactly when $A_\Gamma(K)$ has a character that
annihilates every $z^\beta$, $\beta\in G_\Gamma$, that is, when the
spectrum of $A_\Gamma(K)$ is strictly larger than $K$. By
Theorem~\ref{thm:A} this does not happen for rational polyhedral cones,
and by Corollary~\ref{cor:dimtwo} below it does not happen when $n=2$. We
do not know whether it can happen for $n\ge3$.
\end{remark}

Condition (iii) is often the easiest to verify by hand, and it gives
explicit bounds through \eqref{eq:coordbound}.

\begin{example}\label{ex:certificate}
Let
\[
 \Gamma_0=\Bigl\{x\in\R^3_+:\ x_3\le x_1+x_2+\sqrt{x_1x_2}\Bigr\}.
\]
The function $x_1+x_2+\sqrt{x_1x_2}$ is concave and positively
homogeneous, so $\Gamma_0$ is a closed convex cone with nonempty
interior. It is not polyhedral, and therefore $\cP^{\Gamma_0}$ is not
finitely generated: otherwise the closed cone generated by $M_{\Gamma_0}$,
which equals $\Gamma_0$ because the rational points are dense in
$\Gamma_0$, would be polyhedral. The exponents $e_1$, $e_2$, $e_1+e_3$ and
$e_2+e_3$ lie in $\Gamma_0$, and $e_1,e_2\in G_{\Gamma_0}$; for instance
$e_1$, $2e_1$, $e_1+e_2$ and $e_1+e_3$ all satisfy the defining inequality.

Let $r>0$ and $s>0$, let
\[
 \mathbb B_r=\bigl\{(z_1,z_2)\in\C^2:\ |z_1-1|^2+|z_2-1|^2\le r^2\bigr\}
\]
be the closed Euclidean ball in $\C^2$ with centre $(1,1)$ and radius $r$,
and put
\begin{equation}\label{eq:Krs}
 K_{r,s}=\bigl\{z\in\C^3:\ (z_1,z_2)\in\mathbb B_r,\ |z_1z_3|\le s,\
 |z_2z_3|\le s\bigr\}.
\end{equation}
We assume
\[
 1<r<\sqrt2 ,
\]
and the two bounds play different roles. The distance from $(1,1)$ to
$(0,0)$ in $\C^2$ is $\sqrt2$, so $r<\sqrt2$ means exactly that
$(0,0)\notin\mathbb B_r$. Since $\mathbb B_r$ is compact, there is then
$c>0$ with $\max(|z_1|,|z_2|)\ge c$ on $\mathbb B_r$, and the last two
inequalities in \eqref{eq:Krs} give $|z_3|\le s/c$ on $K_{r,s}$; hence
$K_{r,s}$ is closed and bounded, that is, compact. On the other hand,
$r>1$ means that $\mathbb B_r$ meets the line $\{z_1=0\}$: the point
$(0,z_2)$ lies in $\mathbb B_r$ when $|z_2-1|^2\le r^2-1$. Thus $K_{r,s}$
contains the points $(0,z_2,0)$ with $|z_2-1|^2\le r^2-1$ and,
symmetrically, points $(z_1,0,0)$; in particular it meets the hyperplanes
$\{z_1=0\}$, $\{z_2=0\}$ and $\{z_3=0\}$, but not the axis
$\{z_1=z_2=0\}$.

We show that $K_{r,s}$ is $\cB$-convex for
\[
 \cB=\C[z_1,\,z_2,\,z_1z_3,\,z_2z_3]\subset\cP^{\Gamma_0}.
\]
First, the projection of $K_{r,s}$ to the first two coordinates is all of
$\mathbb B_r$, because $(z_1,z_2,0)\in K_{r,s}$ for every
$(z_1,z_2)\in\mathbb B_r$. Second, $\mathbb B_r$ is polynomially convex:
if $w\notin\mathbb B_r$, the Hahn--Banach theorem gives a complex linear
function $\ell$ on $\C^2$ with
$\operatorname{Re}\ell(w)>\max_{\mathbb B_r}\operatorname{Re}\ell$, and
then $|e^{\ell(w)}|>\max_{\mathbb B_r}|e^\ell|$, where $e^\ell$ is a
uniform limit of polynomials on compact sets. Now let
$z\in\widehat{K_{r,s}}^{\,\cB}$. Every polynomial $q(z_1,z_2)$ belongs to
$\cB$, so
\[
 |q(z_1,z_2)|\le\|q\|_{K_{r,s}}=\|q\|_{\mathbb B_r},
\]
and therefore $(z_1,z_2)\in\widehat{\mathbb B_r}=\mathbb B_r$. Moreover
\[
 |z_1z_3|\le\|z_1z_3\|_{K_{r,s}}\le s,\qquad
 |z_2z_3|\le\|z_2z_3\|_{K_{r,s}}\le s .
\]
Hence $z\in K_{r,s}$, that is, $\widehat{K_{r,s}}^{\,\cB}=K_{r,s}$. Since
$\cB\subset\cP^{\Gamma_0}$, the $\cP^{\Gamma_0}$-hull is contained in the
$\cB$-hull, and $K_{r,s}$ is also $\Gamma_0$-convex. In particular,
condition (v) of Theorem~\ref{thm:B} holds with $\cB_0=\cB$.

Condition (iii) can also be checked by hand, and this is where $r<\sqrt2$
enters a second time. With $q_1=q_2=1/2$ and $\beta^1=e_1$,
$\beta^2=e_2$, the Cauchy--Schwarz inequality gives
\[
 \Bigl\|\,1-\frac{z_1+z_2}2\Bigr\|_{K_{r,s}}
 =\frac12\,\bigl\|(1-z_1)+(1-z_2)\bigr\|_{K_{r,s}}
 \le\frac1{\sqrt2}\,\sup_{\mathbb B_r}\bigl(|1-z_1|^2+|1-z_2|^2\bigr)^{1/2}
 =\frac r{\sqrt2}<1 .
\]
Either way, the Oka--Weil property holds on $K_{r,s}$. Neither $|z_1|$ nor
$|z_2|$ is bounded below on $K_{r,s}$, but their maximum is.
\end{example}

\section{Cases in which the finite-stage condition holds}\label{sec:cases}

We first treat compact sets in the torus. For a compact Reinhardt set the
hull can be computed by an elementary argument, which also explains the
role of the lattice points of $\Gamma$. Let $\Gamma_M$ be the closed convex
cone generated by $M_\Gamma$, let
\[
 \Gamma_M^\vee=\bigl\{v\in\R^n:\ \langle v,s\rangle\ge0\ \text{for all}\
 s\in\Gamma_M\bigr\}
\]
be its dual cone, and let $h_A(s)=\max_{a\in A}\langle s,a\rangle$ be the
support function of a compact set $A\subset\R^n$. We write
$\Log z=(\log|z_1|,\ldots,\log|z_n|)$ for $z\in(\C^*)^n$.

\begin{proposition}\label{prop:reinhardt}
Let $\Gamma\subset\R^n_+$ be a convex cone and let $K\subset(\C^*)^n$ be
a compact Reinhardt set with $A=\Log K$. For $z\in(\C^*)^n$ the following
are equivalent:
\begin{enumerate}[label=\textup{(\arabic*)}]
\item $z\in\widehat K^\Gamma$;
\item $|z^\alpha|\le\|z^\alpha\|_K$ for all $\alpha\in M_\Gamma$;
\item $\Log z\in\conv(A)-\Gamma_M^\vee$.
\end{enumerate}
\end{proposition}

\begin{proof}
Clearly (1) implies (2). Assume (2) and let
$p=\sum c_\alpha z^\alpha\in\cP^\Gamma$. For $w\in K$ the torus
$\{(w_1e^{i\theta_1},\ldots,w_ne^{i\theta_n})\}$ lies in $K$, and the Cauchy formula
\[
 c_\alpha w^\alpha=\frac1{(2\pi)^n}\int_{[0,2\pi]^n}
 p\bigl(w_1e^{i\theta_1},\ldots,w_ne^{i\theta_n}\bigr)\,
 e^{-i\langle\alpha,\theta\rangle}\,d\theta
\]
gives $|c_\alpha|\,|w^\alpha|\le\|p\|_K$. Hence
\[
 |c_\alpha|\,\|z^\alpha\|_K\le\|p\|_K
 \qquad\text{and}\qquad
 |p(z)|\le\sum_\alpha|c_\alpha|\,|z^\alpha|\le N(p)\,\|p\|_K,
\]
where $N(p)$ is the number of monomials of $p$. Applying this to $p^k\in\cP^\Gamma$, whose number of monomials is at most
$\binom{k+N-1}{N-1}\le(k+1)^{N-1}$ when $p$ has $N$ monomials, we get
\[
 |p(z)|^k\le(k+1)^{N-1}\,\|p\|_K^k ,
\]
and letting $k\to\infty$ after taking $k$-th roots gives
$|p(z)|\le\|p\|_K$. This use of high powers is the standard device for
passing from estimates of coefficients to estimates of values.

For (2)$\Leftrightarrow$(3) note that $\|z^\alpha\|_K=e^{h_A(\alpha)}$,
so that (2) says
\[
 \langle\alpha,x\rangle\le h_A(\alpha)\qquad
 \text{for } x=\Log z \text{ and all } \alpha\in M_\Gamma .
\]
This inequality extends to $\Gamma_M$. Indeed, if
$s=\sum_l(a_l/b)\alpha_l$ with $\alpha_l\in M_\Gamma$ and nonnegative
integers $a_l$, $b\ge1$, then $bs\in M_\Gamma$ because $M_\Gamma$ is a
semigroup, and the positive homogeneity of $h_A$ gives
$\langle s,x\rangle=b^{-1}\langle bs,x\rangle\le b^{-1}h_A(bs)=h_A(s)$;
such $s$ are dense in $\Gamma_M$, and both sides are continuous in $s$. The closed
convex set $\conv(A)-\Gamma_M^\vee$ has support function $h_A$ on
$\Gamma_M$ and $+\infty$ elsewhere, because
$\Gamma_M^{\vee\vee}=\Gamma_M$; hence
\[
 \conv(A)-\Gamma_M^\vee=\bigl\{x\in\R^n:\ \langle s,x\rangle\le h_A(s)\
 \text{for all}\ s\in\Gamma_M\bigr\}. \qedhere
\]
\end{proof}

This recovers the description of Reinhardt hulls in
\cite[Corollary~3.3 and Proposition~3.4]{Sig}, the latter being a
modification of \cite[Proposition~5.6]{MSSS}, without extremal functions;
the convex-geometric part of the proof is the same as there.
There the cone is $\R_+S$ and $S$ is normalized so that
$S=\overline{S\cap\Q^n}$; the cone $\Gamma_M$ makes this normalization
automatic.

\begin{example}\label{ex:irrational}
Let $S=\conv\{0,(1,\sqrt2)\}$. The ray $\R_+S$ contains no nonzero lattice
point, so that
\[
 \cP^S=\C,\qquad \widehat{\T^2}^{\,S}=\C^2,
\]
in accordance with $\Gamma_M=\{0\}$. The cone $\R_+S$ itself would give
the half-plane $\{x_1+\sqrt2\,x_2\le0\}$ in logarithmic coordinates.
\end{example}

\begin{corollary}\label{cor:torus}
Let $S\subset\R^n_+$ be compact and convex with $0\in S$, assume that
$\Gamma_S=\R_+S$ has nonempty interior, and let $K\subset(\C^*)^n$ be a
compact $S$-convex set. Then the conditions of Theorem~\ref{thm:B} hold
for $K$. More precisely, there are $\beta\in G_S$ with all coordinates
positive and $q\in\cP^S$ such that $\|1-z^\beta q\|_K<1$.
\end{corollary}

\begin{proof}
Let $\gamma$ be a lattice point in the interior of $\Gamma_S$, which
exists because the interior is a nonempty open cone; its coordinates are
positive because $\Gamma_S\subset\R^n_+$. The same choice of an interior
lattice point is used in \cite[Proposition~2.1]{Sig}. Then $k\gamma+e_j\in\Gamma_S$ for all $j$ and
all large $k$, so $\beta=k\gamma\in G_S$ for a suitable $k$. The function
$z^{-\beta}$ is holomorphic in a neighbourhood of $K$, and Sigurðardóttir's
theorem \cite[Theorem~1.1]{Sig} gives $q\in\cP^S$ with
$\|z^{-\beta}-q\|_K<1/\|z^\beta\|_K$. Then
\[
 \|1-z^\beta q\|_K\le\|z^\beta\|_K\,\|z^{-\beta}-q\|_K<1,
\]
which is condition~(iii) of Theorem~\ref{thm:B}.
\end{proof}

Corollary~\ref{cor:torus} adds nothing to the approximation theorem on
the torus, but it supplies the finite-stage compactness on which the
following arguments rely. Outside the torus we use coordinate directions
along which the compact set stays away from the origin. For a nonempty
proper subset $I\subset\{1,\ldots,n\}$ we write $\pr_I(z)=(z_i)_{i\in I}$
and consider the following conditions on a convex cone $\Gamma$:
\begin{enumerate}[label=\textup{(A\arabic*)}]
\item $e_i\in\Gamma$ for every $i\in I$, and for all $i\in I$ and
$k\notin I$ there is an integer $m_{ik}\ge1$ with
$m_{ik}e_i+e_k\in\Gamma$;
\item every $\alpha\in M_\Gamma\setminus\{0\}$ satisfies
$\alpha-\delta e_i\in\Gamma$ for some $i\in I$ and some $\delta>0$.
\end{enumerate}
Condition (A2) says that no nonzero lattice point of $\Gamma$ lies on the
part of the boundary of $\Gamma$ that is reached by moving in one of the
directions $-e_i$, $i\in I$. The way these conditions are used, namely
first to bound the anchor coordinates $z_i$, $i\in I$, from below and then
to bound the other coordinates through the monomials $z_i^{m_{ik}}z_k$, is
similar in spirit to the argument showing that monomial maps are proper on
the torus \cite[Lemma~2.3 and Proposition~2.4]{Sig}.

\begin{theorem}\label{thm:anchor}
Let $\Gamma\subset\R^n_+$ be a convex cone, let
$\varnothing\ne I\subsetneq\{1,\ldots,n\}$, assume \textup{(A1)}, and let
$K\subset\C^n$ be compact.
\begin{enumerate}[label=\textup{(\alph*)}]
\item If $0\notin\widehat{\pr_I(K)}$, the ordinary polynomial hull in
$\C^I$, then $\widehat K^\cB$ is compact for
\[
 \cB=\C\bigl[z_i,\ z_i^{m_{ik}}z_k:\ i\in I,\ k\notin I\bigr].
\]
\item If $\Gamma$ satisfies \textup{(A2)} and $K$ is $\Gamma$-convex, then
$\widehat K^\cB$ is compact for some finitely generated monomial algebra
$\cB\subset\cP^\Gamma$.
\end{enumerate}
In both cases, if $K$ is $\Gamma$-convex, the Oka--Weil property holds on
$K$.
\end{theorem}

\begin{proof}
In both parts we show that, for a suitable finitely generated algebra
$\cB\supset\C[z_i,z_i^{m_{ik}}z_k]$ and some $\varepsilon>0$,
\begin{equation}\label{eq:anchorbound}
 \rho(z):=\max_{i\in I}|z_i|\ge\varepsilon\qquad
 \text{for all }z\in\widehat K^\cB .
\end{equation}
This suffices. On $\widehat K^\cB$ we have $|z_i|\le\|z_i\|_K$ for
$i\in I$, and for $k\notin I$ we choose $i\in I$ with
$|z_i|\ge\varepsilon$ and obtain
\[
 |z_k|\le\varepsilon^{-m_{ik}}\,\|z_i^{m_{ik}}z_k\|_K .
\]
Thus the closed set $\widehat K^\cB$ is bounded, and Theorem~\ref{thm:B}
gives the Oka--Weil property when $K$ is $\Gamma$-convex.

(a) The algebra $\cB$ contains all polynomials in $(z_i)_{i\in I}$, so
\[
 \pr_I\bigl(\widehat K^\cB\bigr)\subset\widehat{\pr_I(K)} .
\]
The latter is a compact set not containing $0$, which gives
\eqref{eq:anchorbound}.

(b) By (A2), every nonzero $\alpha\in M_\Gamma$ has a positive coordinate
indexed by $I$. Hence every element of $\cP^\Gamma$ is constant on the
unbounded set
\[
 V_I=\{z\in\C^n:\ z_i=0\ \text{for}\ i\in I\},
\]
and the compact $\Gamma$-convex set $K$ does not meet $V_I$; in
particular $0\notin K$. Choose
\[
 p=c+\sum_{\alpha\in F}c_\alpha z^\alpha\in\cP^\Gamma,\qquad
 F\subset M_\Gamma\setminus\{0\}\ \text{finite},\qquad
 |c|=|p(0)|>\|p\|_K .
\]
For $\alpha\in F$ choose $i(\alpha)\in I$ and $\delta_\alpha>0$ as in
(A2), and integers $N_\alpha,j_\alpha\ge1$ with
$j_\alpha/N_\alpha<\delta_\alpha$. The point
$\alpha-(j_\alpha/N_\alpha)e_{i(\alpha)}$ lies on the segment between
$\alpha$ and $\alpha-\delta_\alpha e_{i(\alpha)}$, hence in $\Gamma$, and
therefore
\[
 \gamma_\alpha:=N_\alpha\alpha-j_\alpha e_{i(\alpha)}
 =N_\alpha\Bigl(\alpha-\frac{j_\alpha}{N_\alpha}\,e_{i(\alpha)}\Bigr)
 \in M_\Gamma .
\]
Let $\cB$ be generated by the $z_i$, the $z_i^{m_{ik}}z_k$, and the
monomials $z^\alpha$ and $z^{\gamma_\alpha}$ for $\alpha\in F$. The
identity $N_\alpha\alpha=j_\alpha e_{i(\alpha)}+\gamma_\alpha$ gives, for
$z\in\widehat K^\cB$,
\[
 |z^\alpha|
 \le|z_{i(\alpha)}|^{j_\alpha/N_\alpha}\,
 \|z^{\gamma_\alpha}\|_K^{1/N_\alpha}
 \le\rho(z)^{j_\alpha/N_\alpha}\,\|z^{\gamma_\alpha}\|_K^{1/N_\alpha}.
\]
Consequently, for $z\in\widehat K^\cB$,
\[
 |p(z)-c|\le\sum_{\alpha\in F}|c_\alpha|\,
 \rho(z)^{j_\alpha/N_\alpha}\,\|z^{\gamma_\alpha}\|_K^{1/N_\alpha}.
\]
The right-hand side tends to $0$ as $\rho(z)\to0$, so there is
$\varepsilon>0$ such that it is smaller than $|c|-\|p\|_K$ whenever
$\rho(z)<\varepsilon$. For such $z$ the reverse triangle inequality gives
$|p(z)|>\|p\|_K$, which is impossible for $z\in\widehat K^\cB$ because
$p\in\cB$. Hence $\rho\ge\varepsilon$ on $\widehat K^\cB$, which is
\eqref{eq:anchorbound}.
\end{proof}

\begin{example}\label{ex:pi}
For $n\ge2$ let
\[
 \Gamma_\pi=\Bigl\{x\in\R^n_+:\
 x_1\ge\pi\bigl(x_2^2+\cdots+x_n^2\bigr)^{1/2}\Bigr\},
\]
a closed convex cone with nonempty interior, nonpolyhedral for $n\ge3$, and
equal to $\R_+S$ for $S=\Gamma_\pi\cap\{x_1+\cdots+x_n\le1\}$. If
$\alpha\in M_{\Gamma_\pi}$ and $\alpha_j\ne0$ for some $j\ge2$, then
equality in the defining inequality would make $\pi^2$ rational; hence
the inequality is strict and $\alpha-\delta e_1\in\Gamma_\pi$ for small
$\delta>0$, which is also clear for $\alpha\in\N e_1$. Moreover $e_1$ and
$4e_1+e_j$ lie in $\Gamma_\pi$. Thus (A1) and (A2) hold with $I=\{1\}$, and
the Oka--Weil property holds on every compact $\Gamma_\pi$-convex subset of
$\C^n$. For instance, if $|c|>r>0$ and $s_j>0$, then
\[
 \bigl\{z\in\C^n:\ |z_1-c|\le r,\ |z_1^4z_j|\le s_j\ (2\le j\le n)\bigr\}
\]
is such a set, since the map $(z_1,z_1^4z_2,\ldots,z_1^4z_n)$ identifies it
with a closed polydisc; it meets every coordinate hyperplane
$\{z_j=0\}$, $j\ge2$, and all their intersections.
\end{example}

For the cone $\Gamma_0$ of Example~\ref{ex:certificate}, condition (A2)
fails for $I=\{1,2\}$: the lattice point $(1,0,1)$ lies on the boundary,
and neither $(1,0,1)-\delta e_1$ nor $(1,0,1)-\delta e_2$ belongs to
$\Gamma_0$ for $\delta>0$. Theorem~\ref{thm:anchor}(a) still applies to
the set $K_{r,s}$ of \eqref{eq:Krs}, because $\pr_{\{1,2\}}(K_{r,s})$ is a
closed Euclidean ball not containing the origin; this gives a second proof
of the conclusion of Example~\ref{ex:certificate}.

Condition (A1) requires coordinate vectors in $\Gamma$. When $\Gamma$
contains none, one can instead use a block of coordinates on which the
compact set stays in the torus, and transfer
Corollary~\ref{cor:torus} from that block. For $J\subset\{1,\ldots,n\}$ we
write $S_J=S\cap\R^J$; then $\R_+S_J=\Gamma_S\cap\R^J$, and we regard
$\cP^{S_J}$ as a ring of polynomials on $\C^J$.

\begin{proposition}[Block criterion]\label{prop:block}
Let $S\subset\R^n_+$ be compact and convex with $0\in S$, let
$\Gamma=\Gamma_S$, and let $\varnothing\ne J\subsetneq\{1,\ldots,n\}$.
Assume that
\begin{enumerate}[label=\textup{(B\arabic*)}]
\item $M_J=M_\Gamma\cap\R^J$ spans $\R^J$;
\item for every $k\notin J$ there is $\alpha^{(k)}\in M_\Gamma$ supported
in $J\cup\{k\}$ with $\alpha^{(k)}_k>0$.
\end{enumerate}
Let $K\subset\C^n$ be compact and assume that the hull $H_J$ of
$\pr_J(K)$ with respect to $\cP^{S_J}$ in $\C^J$ is a compact subset of
$(\C^*)^J$. Then $\widehat K^\cB$ is compact for some finitely generated
monomial algebra $\cB\subset\cP^S$. In particular, if $K$ is $S$-convex,
the Oka--Weil property holds on $K$.
\end{proposition}

\begin{proof}
The argument combines Corollary~\ref{cor:torus} with lower bounds for
coordinates obtained from monomial inequalities, as in
\cite[Lemma~2.3]{Sig}. By (B1) the cone $\R_+S_J$ has nonempty interior in
$\R^J$. The set $H_J$
is a compact $S_J$-convex subset of $(\C^*)^J$, and
Corollary~\ref{cor:torus}, applied in $\C^J$, gives $\beta\in G_{S_J}$
with all coordinates positive and $q\in\cP^{S_J}$ with
\[
 \vartheta:=\|1-z^\beta q\|_{H_J}<1 .
\]
Let $\cB_J$ be generated by the monomials of $q$ and by $z^\beta$ and
$z^{\beta+e_j}$, $j\in J$. Since $\|b\|_{H_J}=\|b\|_{\pr_J(K)}$ for
$b\in\cP^{S_J}$, every $w$ in the $\cB_J$-hull $H'$ of $\pr_J(K)$ satisfies
\[
 |1-w^\beta q(w)|\le\vartheta,\qquad |q(w)|\le\|q\|_{H_J},\qquad
 |w^{\beta+e_j}|\le\|z^{\beta+e_j}\|_{H_J},
\]
and hence
\[
 |w^\beta|\ge\frac{1-\vartheta}{\|q\|_{H_J}}>0,\qquad
 |w_j|=\frac{|w^{\beta+e_j}|}{|w^\beta|}\le
 \frac{\|q\|_{H_J}\,\|z^{\beta+e_j}\|_{H_J}}{1-\vartheta}
 \quad(j\in J).
\]
Since all coordinates of $\beta$ are positive, these bounds give, for
each $j\in J$,
\[
 |w_j|^{\beta_j}=\frac{|w^\beta|}{\prod_{l\in J,\,l\ne j}|w_l|^{\beta_l}}
 \ge\frac{1-\vartheta}{\|q\|_{H_J}}
 \prod_{l\in J,\,l\ne j}\Bigl(\frac{1-\vartheta}
 {\|q\|_{H_J}\,\|z^{\beta+e_l}\|_{H_J}}\Bigr)^{\beta_l}>0 .
\] Thus $H'$ is a compact subset of $(\C^*)^J$.

Let $\cB\subset\cP^S$ be generated by the monomials of $\cB_J$, regarded
as polynomials on $\C^n$, and by $z^{\alpha^{(k)}}$, $k\notin J$. If
$z\in\widehat K^\cB$, then $\pr_J(z)\in H'$, so that
$0<c\le|z_j|\le C$ for $j\in J$ with constants independent of $z$. For
$k\notin J$ the inequality $|z^{\alpha^{(k)}}|\le\|z^{\alpha^{(k)}}\|_K$
gives
\[
 |z_k|^{\alpha^{(k)}_k}\le
 \frac{\|z^{\alpha^{(k)}}\|_K}{\prod_{j\in J}|z_j|^{\alpha^{(k)}_j}}
 \le\frac{\|z^{\alpha^{(k)}}\|_K}{\prod_{j\in J}c^{\,\alpha^{(k)}_j}} .
\]
Hence $\widehat K^\cB$ is bounded, and it is closed. The last assertion
follows from Theorem~\ref{thm:B}.
\end{proof}

\begin{example}\label{ex:block}
In $\R^3_+$ let
\[
 \Gamma=\Bigl\{x\in\R^3_+:\ 2x_2\le x_1\le3x_2,\
 x_3\le\sqrt{x_1x_2}\Bigr\},
\]
a closed convex cone with nonempty interior which is not polyhedral and
equals $\R_+S$ for $S=\Gamma\cap\{x_1+x_2+x_3\le1\}$. It contains none of
$e_1,e_2,e_3$, so condition (A1) fails for every $I$. For $J=\{1,2\}$ the
lattice points
\[
 \alpha^1=(2,1,0),\qquad \alpha^2=(3,1,0),\qquad \alpha^{(3)}=(2,1,1)
\]
belong to $\Gamma$; the first two span $\R^J$, so (B1) and (B2) hold. Put
\[
 u=z_1^2z_2,\qquad v=z_1^3z_2,\qquad w=z_1^2z_2z_3 .
\]
On $\{z_1z_2\ne0\}$ the map $(u,v,w)$ is a bijection onto
$\{uv\ne0\}\times\C$, with inverse
\[
 z_1=\frac vu,\qquad z_2=\frac{u^3}{v^2},\qquad z_3=\frac wu .
\]
Fix $c_1,c_2\in\C^*$ and $0<r_j<|c_j|$, and let
\[
 Q=\bigl\{(z_1,z_2)\in\C^2:\ |u-c_1|\le r_1,\ |v-c_2|\le r_2\bigr\}.
\]
Then $Q$ is a compact $S_J$-convex subset of $(\C^*)^2$, since $(u,v)$
identifies it with a closed bidisc in $(\C^*)^2$. Consequently, every
compact $\Gamma$-convex set $K\subset\C^3$ with $\pr_J(K)\subset Q$ satisfies
the hypothesis of Proposition~\ref{prop:block}, because $H_J\subset Q$,
and the Oka--Weil property holds on $K$. For instance,
\[
 K=\bigl\{z\in\C^3:\ |u-c_1|\le r_1,\ |v-c_2|\le r_2,\ |w|\le s\bigr\}
\]
is such a set; the map $(u,v,w)$ identifies it with a closed polydisc, and
it meets the plane $\{z_3=0\}$.
\end{example}

In two variables the conditions (A1)--(A2) cover everything that is not
already rational polyhedral or contained in the torus.

\begin{corollary}\label{cor:dimtwo}
Let $S\subset\R^2_+$ be compact and convex with $0\in S$. Then the
conditions of Theorem~\ref{thm:B} hold, and the Oka--Weil property holds,
for every compact $S$-convex set $K\subset\C^2$.
\end{corollary}

\begin{proof}
We may assume $K\ne\varnothing$, so that $\Gamma=\Gamma_S$ has nonempty
interior by Corollary~\ref{cor:goodcover}. If $e_1,e_2\in\Gamma$, then
$\Gamma=\R^2_+$ and the classical Oka--Weil theorem applies. If
$e_1,e_2\notin\Gamma$, then every nonzero $\alpha\in M_\Gamma$ has two
positive coordinates, every good exponent is nonzero, $U_\Gamma\subset
(\C^*)^2$, and Corollaries~\ref{cor:goodcover} and~\ref{cor:torus} apply.

By symmetry it remains to consider the case $e_1\in\Gamma$,
$e_2\notin\Gamma$. Then $x_1>0$ for every $x\in\Gamma\setminus\{0\}$, and
we put
\[
 \sigma=\sup\Bigl\{\frac{x_2}{x_1}:\ x\in\Gamma\setminus\{0\}\Bigr\}
 \in(0,\infty].
\]
For $0\le\tau<\sigma$ there is $x\in\Gamma$ with $\sigma'=x_2/x_1>\tau$,
and
\[
 (1,\tau)=\Bigl(1-\frac\tau{\sigma'}\Bigr)e_1+\frac\tau{\sigma'}\cdot
 \frac x{x_1}
\]
is a nonnegative combination of $e_1$ and $x$, so $(1,\tau)\in\Gamma$. If $\Gamma$ is generated by $e_1$ and a lattice
point, then $\Gamma$ is rational polyhedral, and Theorems~\ref{thm:A}
and~\ref{thm:B} apply. Otherwise we verify (A1) and (A2) for $I=\{1\}$.
Condition (A1) holds because
\[
 me_1+e_2=m\Bigl(1,\frac1m\Bigr)\in\Gamma\qquad\text{when }\frac1m<\sigma .
\]
Let $\alpha=(a,b)\in M_\Gamma\setminus\{0\}$, so that $a\ge1$ and
$b/a\le\sigma$. If $b/a<\sigma$, choose $0<\delta<a$ with
$b/(a-\delta)<\sigma$; then
\[
 \alpha-\delta e_1=(a-\delta)\Bigl(1,\frac b{a-\delta}\Bigr)\in\Gamma .
\]
If $b/a=\sigma$, then $\Gamma$ lies in $\{x_2\le\sigma x_1\}\cap\R^2_+$,
which is the cone generated by $e_1$ and the lattice point $\alpha$; since
both lie in $\Gamma$, this is the excluded rational polyhedral case.
Theorem~\ref{thm:anchor}(b) completes the proof.
\end{proof}

\begin{example}\label{ex:nonclosed}
Let
\[
 S=\bigl\{x\in\R^2_+:\ 0\le x_1\le1,\ x_2\le x_1(1-x_1)\bigr\},
\]
which is compact and convex because $x_1(1-x_1)$ is concave. A point
$(x_1,\tau x_1)$ with $x_1>0$ lies in $S$ exactly when $x_1\le1-\tau$, so
\[
 \R_+S=\{0\}\cup\{x\in\R^2:\ x_1>0,\ 0\le x_2<x_1\},\qquad
 M_{\R_+S}=\{0\}\cup\{(a,b)\in\N^2:\ a>b\}.
\]
The cone is not closed, and $\cP^S$ is not finitely generated, since the
lattice points $(k+1,k)$ are indecomposable in $M_{\R_+S}$. By
Corollary~\ref{cor:dimtwo} the Oka--Weil theorem holds on every compact
$S$-convex subset of $\C^2$. An example meeting the coordinate axis
$\{z_2=0\}$ is
\[
 \bigl\{z\in\C^2:\ |z_1-2|\le\tfrac12,\ |z_1^2z_2|\le1\bigr\},
\]
which the map $(z_1,z_2)\mapsto(z_1,z_1^2z_2)$ identifies with a closed
bidisc.
\end{example}

A compact set lying in a coordinate subspace can be treated in that
subspace. We write $\C^J=\{z\in\C^n:\ z_i=0\text{ for }i\notin J\}$,
identified with $\C^{|J|}$. The restriction $p\mapsto p|_{\C^J}$ maps
$\cP^S$ onto $\cP^{S_J}$, with a right inverse given by regarding a
polynomial on $\C^J$ as a polynomial on $\C^n$ independent of the
variables outside $J$.

The following proposition is the approximation counterpart of the
restriction of hulls and extremal functions to coordinate strata in
\cite[Proposition~3.5]{Sig}.

\begin{proposition}\label{prop:face}
Let $S\subset\R^n_+$ be compact and convex with $0\in S$, let
$J\subset\{1,\ldots,n\}$ be nonempty, and let $K\subset\C^J$ be a compact
$S$-convex set. Then $K$ is $S_J$-convex as a subset of $\C^J$, and if the
Oka--Weil property for $\cP^{S_J}$ holds on $K$ in $\C^J$, then the
Oka--Weil property for $\cP^S$ holds on $K$ in $\C^n$. In particular, the
Oka--Weil property holds on every compact $S$-convex subset of a
coordinate subspace of dimension at most two.
\end{proposition}

\begin{proof}
For $p\in\cP^S$ we have $\|p|_{\C^J}\|_K=\|p\|_K$. If $w$ lies in the
$\cP^{S_J}$-hull of $K$ in $\C^J$, then
\[
 |p(w)|\le\|p\|_K\qquad\text{for all }p\in\cP^S,
\]
so $w\in\widehat K^S=K$. If $f$ is holomorphic near $K$ in $\C^n$, then
$f|_{\C^J}$ is holomorphic near $K$ in $\C^J$; approximating it on $K$ by
elements of $\cP^{S_J}$ and regarding these as elements of $\cP^S$ gives
the second assertion. If $|J|=2$, Corollary~\ref{cor:dimtwo} applies in
$\C^J$. If $|J|=1$, say $J=\{j\}$, Corollary~\ref{cor:goodcover} in
$\C^J$ shows that $\R_+S_J$ is the ray $\R_+e_j$, so that
$\cP^{S_J}=\C[z_j]$ and the classical one-variable theorem applies.
\end{proof}

Theorem~\ref{thmC} follows from Corollaries~\ref{cor:dimtwo}
and~\ref{cor:torus}, Theorem~\ref{thm:anchor}(b) and
Proposition~\ref{prop:face}.

\section{Applications: Oka--Stolzenberg criteria and rational
approximation}\label{sec:applications}

Stolzenberg \cite[(1.2)]{Stolz} showed that a point $y$ lies outside the
polynomial hull of a compact set $L\subset\C^N$ if and only if $y$ lies on
an algebraic hypersurface that can be moved continuously to infinity
without meeting $L$, and that $y$ lies outside the rational hull if and
only if a single algebraic hypersurface passes through $y$ and misses $L$
\cite[(1.1)]{Stolz}. The second statement is immediate from the
definition, and its analogue for $\cP^S$ is the description of
$\cR_S(K)$ given in Section~\ref{sec:prelim}. The first statement is a
consequence of the Oka--Weil theorem, and its sparse analogue follows from
the results above. We say that a family $(p_t)_{t\ge0}$ of polynomials is
\emph{continuous} if $(t,w)\mapsto p_t(w)$ is continuous on
$[0,\infty)\times\C^n$, and that the hypersurfaces
$Z_t=\{p_t=0\}$ \emph{escape} if for every compact $B\subset\C^n$
\[
 Z_t\cap B=\varnothing\qquad\text{for all sufficiently large }t .
\]

The proof of the next theorem is Stolzenberg's proof of
\cite[(1.2)]{Stolz}; the only change is that the Oka--Weil theorem on the
polynomial hull is replaced by the Oka--Weil property on the
$\cP^\Gamma$-hull, which is available by Theorem~\ref{thm:B}.

\begin{theorem}\label{thm:OS}
Let $\Gamma\subset\R^n_+$ be a convex cone and let $K\subset\C^n$ be a
compact set such that $\widehat K^{\cB_0}$ is compact for some finitely
generated monomial algebra $\cB_0\subset\cP^\Gamma$. For $z\in\C^n$ the
following are equivalent.
\begin{enumerate}[label=\textup{(\roman*)}]
\item $z\notin\widehat K^\Gamma$.
\item There is a continuous family $(p_t)_{t\ge0}$ in $\cP^\Gamma$ such
that $p_0(z)=0$, no $p_t$ has a zero on $K$, and the hypersurfaces
$\{p_t=0\}$ escape.
\end{enumerate}
In \textup{(ii)} one may take $p_t=p-(1+t)$ with $p\in\cP^\Gamma$; in
particular, all $p_t$ lie in one space $\cP^S_m$ when $\Gamma=\Gamma_S$.
The implication \textup{(i)$\Rightarrow$(ii)} holds for every compact
$K$.
\end{theorem}

\begin{proof}
(i)$\Rightarrow$(ii). Choose $p\in\cP^\Gamma$ with $p(z)=1>\|p\|_K$; this
is possible for every compact $K$. The polynomials
\[
 p_t=p-(1+t),\qquad t\ge0,
\]
satisfy $p_0(z)=0$ and $|p_t|\ge1+t-\|p\|_K>0$ on $K$, and $\{p_t=0\}$
escapes because $p$ is bounded on compact sets. Since $\cP^S_m$ contains
the constants, $p_t\in\cP^S_m$ whenever $p\in\cP^S_m$.

(ii)$\Rightarrow$(i). By Remark~\ref{rem:nonconvex},
$H=\widehat K^\Gamma$ is compact and the Oka--Weil property holds on $H$.
Suppose that $z\in H$ and let
\[
 T=\bigl\{t\ge0:\ \{p_t=0\}\cap H\ne\varnothing\bigr\}.
\]
Then $0\in T$ because $z\in H$, and $T$ is bounded because the
hypersurfaces escape and $H$ is compact. Moreover $T$ is closed: if
$t_l\in T$, $t_l\to t$, and $y_l\in H$ satisfy $p_{t_l}(y_l)=0$, then a
subsequence of $(y_l)$ converges to some $y\in H$, and $p_t(y)=0$ by the
joint continuity of $(t,w)\mapsto p_t(w)$. Let $t_*=\max T$ and choose $y\in H$ with
$p_{t_*}(y)=0$. For $t>t_*$ the function $1/p_t$ is holomorphic in a
neighbourhood of $H$, so it is a uniform limit on $H$ of elements
$a_j\in\cP^\Gamma$. Since $y\in H$ and $|a_j(y)|\le\|a_j\|_K$, we get
\[
 \frac1{|p_t(y)|}\le\sup_K\frac1{|p_t|}\qquad\text{for all }t>t_* .
\]
The right-hand side is bounded for $t\in(t_*,t_*+1]$, because $p_t(w)\ne0$
on the compact set $[t_*,t_*+1]\times K$, whereas the left-hand side tends to $+\infty$ as $t\downarrow t_*$,
because $p_t(y)\to p_{t_*}(y)=0$ and $p_t(y)\ne0$ for $t>t_*$ (indeed
$y\in H$ and $t\notin T$). This contradiction proves (i).
\end{proof}

\begin{remark}\label{rem:OSlocal}
(1) Without a compact finite stage the hull may be unbounded, and the
criterion must be formulated in ambient spaces. If
$\cB_1\subset\cB_2\subset\cdots$ are finitely generated monomial algebras
with union $\cP^\Gamma$ and $\pi_j$ are the corresponding monomial maps,
then
\[
 \widehat K^\Gamma=\bigcap_j\pi_j^{-1}\bigl(\widehat{\pi_j(K)}\bigr),
\]
and Stolzenberg's criterion applies to each $\pi_j(K)$ in $\C^{N_j}$. The
preimages of escaping hypersurfaces in $\C^{N_j}$ need not escape in
$\C^n$, because $\pi_j$ may have unbounded fibres.

(2) For a compact Reinhardt set $K\subset(\C^*)^n$ and
$z\in(\C^*)^n\setminus\widehat K^\Gamma$, Proposition~\ref{prop:reinhardt}
gives $\alpha\in M_\Gamma$ with $|z^\alpha|>\|z^\alpha\|_K$, and the
hypersurfaces
\[
 \bigl\{w\in\C^n:\ w^\alpha=(1+t)\,z^\alpha\bigr\},\qquad t\ge0,
\]
form an explicit escaping family of the kind in Theorem~\ref{thm:OS}(ii),
even when $\widehat K^\Gamma$ is unbounded.
\end{remark}

We turn to rational approximation. Rational hulls and their
hypersurface description are due to Stolzenberg \cite[(1.1)]{Stolz}, and
the rational Oka--Weil theorem is classical \cite[Chapter~1]{Stout}. Here
the denominators are allowed to
range over $\cP^\Gamma$, so the degree parameter is not fixed.

\begin{theorem}\label{thm:rational}
Let $\Gamma\subset\R^n_+$ be a convex cone and let $K\subset\C^n$ be a
nonempty compact set with $\cR_\Gamma(K)=K$. Assume that one of the
following conditions holds:
\begin{enumerate}[label=\textup{(\alph*)}]
\item $\Gamma$ is rational polyhedral;
\item $\cR_{\cB_0}(K)$ is compact for some finitely generated monomial
algebra $\cB_0\subset\cP^\Gamma$; this holds in particular when
$\widehat K^{\cB_0}$ is compact;
\item $K\subset(\C^*)^n$.
\end{enumerate}
Then every function holomorphic in a neighbourhood of $K$ is a uniform
limit on $K$ of quotients $p/q$ with $p,q\in\cP^\Gamma$ and $q$ without
zeros on $K$. Moreover, if $\Gamma$ has nonempty interior, then
\[
 \cR_\Gamma(L)=\cR(L)\qquad\text{for every compact }L\subset(\C^*)^n .
\]
\end{theorem}

\begin{proof}
Case (a) is Theorem~\ref{thm:A}(b). In case (b) we repeat the proof of
(v)$\Rightarrow$(i) in Theorem~\ref{thm:B} with
\[
 K_j=\cR_{\cB_j}(K),\qquad j=1,2,\ldots
\]
These sets are closed subsets of the compact set $\cR_{\cB_0}(K)$, they
decrease, and $\bigcap_jK_j=\cR_\Gamma(K)=K$. Moreover
$\cR_{\cB_j}(K_j)=K_j$, because $b(K_j)\subset b(K)$ for $b\in\cB_j$. Let $f$ be holomorphic on an open set $\Omega\supset K$. By
Corollary~\ref{cor:goodcover}, $K$ lies in a finite union $V$ of good
charts $\{z^{\beta^k}\ne0\}$, and we may assume $\Omega\subset V$ and that
every $\cB_j$ contains the monomials $z^{\beta^k}$ and $z^{\beta^k+e_i}$.
As in the proof of Theorem~\ref{thm:B}, for large $j$ we have
$K_j\subset\Omega\subset V\subset U_{\cB_j}$, and
Theorem~\ref{thm:chart}(b), applied to $K_j$ and $\cB_j$, shows that $f$
is a uniform limit on $K_j$, hence on $K$, of quotients $b/c$ with
$b,c\in\cB_j\subset\cP^\Gamma$ and $c\ne0$ on $K_j\supset K$.

For (c) and the last assertion, $\Gamma$ has nonempty interior by
Corollary~\ref{cor:goodcover} in case (c) and by assumption otherwise.
Choose a lattice point $\beta$ in the interior of $\Gamma$; its
coordinates are positive. For a polynomial $p$ and a large integer $k$,
all exponents of $z^{k\beta}p$ lie in $\Gamma$, since
\[
 \alpha+k\beta=k\Bigl(\beta+\frac\alpha k\Bigr).
\]
Let $L\subset(\C^*)^n$ be compact. A point with a vanishing coordinate is
not in $\cR_\Gamma(L)$, since $z^\beta$ vanishes there and not on $L$. If
$w\in(\C^*)^n\setminus\cR(L)$, a polynomial $p$ vanishes at $w$ and not on
$L$, and so does $z^{k\beta}p\in\cP^\Gamma$. Hence
$\cR_\Gamma(L)\subset\cR(L)$, and the opposite inclusion is clear. In case
(c), $K=\cR(K)$ is rationally convex, so a function holomorphic near $K$
is a uniform limit of rational functions $p/q$ with $q\ne0$ on $K$, and
\[
 \frac pq=\frac{z^{k\beta}p}{z^{k\beta}q},\qquad
 z^{k\beta}p,\ z^{k\beta}q\in\cP^\Gamma\ \text{for large }k. \qedhere
\]
\end{proof}

The difference between the two hulls is already visible in the simplest
example.

\begin{example}\label{ex:circle}
Let $S=\conv\{0,e_1,e_1+e_2\}$, so that $\cP^S=\C[z_1,z_1z_2]$, and let
\[
 K=\bigl\{(2,z_2):\ |z_2|=1\bigr\}.
\]
Under $\pi=(z_1,z_1z_2)$ the set $K$ is mapped onto a circle in the line
$\{w_1=2\}$, and \eqref{eq:pullback} gives
\[
 \widehat K^S=\bigl\{(2,z_2):\ |z_2|\le1\bigr\},\qquad \cR_S(K)=K .
\]
The function $1/z_2$ is holomorphic near $K$ and equals the admissible
quotient
\[
 \frac1{z_2}=\frac{z_1}{z_1z_2}
\]
there, but it is not a uniform limit on $K$ of elements of $\cP^S$: such
limits extend holomorphically to the disc $\widehat K^S$ by the maximum
principle. The point $(2,0)$ lies on the single hypersurface
$\{z_1z_2=0\}$, which misses $K$; by Theorem~\ref{thm:OS}, which applies
because $\widehat K^S$ is compact, no escaping family of hypersurfaces
$\{p_t=0\}$, $p_t\in\cP^S$, starts at $(2,0)$ without meeting $K$.
Similarly, for the torus $\T^2$ Theorem~\ref{thm:rational} gives
\[
 \cR_S(\T^2)=\cR(\T^2)=\T^2,
\]
while $\widehat{\T^2}^{\,S}$ is unbounded (Example~\ref{ex:triangle}).
\end{example}


\begin{thebibliography}{99}

\bibitem{BL}
L.~Bos and N.~Levenberg,
\emph{Bernstein--Walsh theory associated to convex bodies and
applications to multivariate approximation theory},
Comput. Methods Funct. Theory \textbf{18} (2018), 361--388.

\bibitem{CLS}
D.~A.~Cox, J.~B.~Little and H.~K.~Schenck,
\emph{Toric Varieties}, Graduate Studies in Mathematics \textbf{124},
American Mathematical Society, Providence, RI, 2011.

\bibitem{Gam}
T.~W.~Gamelin, \emph{Uniform Algebras}, Prentice-Hall, Englewood Cliffs,
NJ, 1969.

\bibitem{Hor}
L.~Hörmander,
\emph{An Introduction to Complex Analysis in Several Variables},
3rd ed., North-Holland, Amsterdam, 1990.

\bibitem{MSS}
B.~S.~Magnússon, Á.~E.~Sigurðardóttir and R.~Sigurðsson,
\emph{Polynomials with exponents in compact convex sets and associated
weighted extremal functions: the Siciak--Zakharyuta theorem},
Complex Anal. Synerg. \textbf{10} (2024), Paper No.~12.

\bibitem{MSSS}
B.~S.~Magnússon, Á.~E.~Sigurðardóttir, R.~Sigurðsson and B.~Snorrason,
\emph{Polynomials with exponents in compact convex sets and associated
weighted extremal functions -- fundamental results},
Ann. Polon. Math. \textbf{133} (2024), 37--70.

\bibitem{MSiS}
B.~S.~Magnússon, R.~Sigurðsson and B.~Snorrason,
\emph{Polynomials with exponents in compact convex sets and associated
weighted extremal functions: the Bernstein--Walsh--Siciak theorem},
arXiv:2306.02486.

\bibitem{Oka}
K.~Oka,
\emph{Sur les fonctions analytiques de plusieurs variables. I.
Domaines convexes par rapport aux fonctions rationnelles},
J. Sci. Hiroshima Univ. Ser.~A \textbf{6} (1936), 245--255.

\bibitem{Sig}
Á.~E.~Sigurðardóttir,
\emph{Holomorphic approximation by polynomials with exponents restricted
to a convex cone}, arXiv:2409.12132v3 (2025).

\bibitem{Stolz}
G.~Stolzenberg,
\emph{Polynomially and rationally convex sets},
Acta Math. \textbf{109} (1963), 259--289.

\bibitem{Stout}
E.~L.~Stout, \emph{Polynomial Convexity}, Progress in Mathematics
\textbf{261}, Birkhäuser, Boston, 2007.

\bibitem{Tref}
L.~N.~Trefethen,
\emph{Multivariate polynomial approximation in the hypercube},
Proc. Amer. Math. Soc. \textbf{145} (2017), 4837--4844.

\bibitem{Weil}
A.~Weil,
\emph{L'intégrale de Cauchy et les fonctions de plusieurs variables},
Math. Ann. \textbf{111} (1935), 178--182.

\end{thebibliography}
\end{document}